\documentclass[11pt, reqno]{amsart}

\usepackage[bookmarksnumbered, colorlinks, plainpages]{hyperref}
\hypersetup{colorlinks=true,linkcolor=red, anchorcolor=green, citecolor=cyan, urlcolor=red, filecolor=magenta, pdftoolbar=true}

\usepackage{amssymb, amscd, latexsym}
\usepackage{amsmath}
\usepackage{amsthm}
\usepackage[capitalize]{cleveref}
\usepackage{float}
\usepackage{pgf, tikz}
\usetikzlibrary{math}
\usetikzlibrary{decorations.pathmorphing, decorations.pathreplacing}
\usepackage{subcaption}

\theoremstyle{plain}
\newtheorem{theorem}{Theorem}[section]
\newtheorem{proposition}[theorem]{Proposition}
\newtheorem{lemma}[theorem]{Lemma}
\newtheorem{corollary}[theorem]{Corollary}
\newtheorem{conjecture}[theorem]{Conjecture}

\theoremstyle{definition}
\newtheorem*{definition}{Definition}

\newtheorem*{question}{Question}

\theoremstyle{remark}
\newtheorem*{remark}{Remark}

\newcommand{\Cay}{\mathrm{Cay}}
\newcommand{\pmd}{\mathrm{pmd}}

\newcommand{\G}{\Gamma}
\newcommand{\calC}{\mathcal{C}}
\newcommand{\calE}{\mathcal{E}}
\newcommand{\calF}{\mathcal{F}}
\newcommand{\calM}{\mathcal{M}}
\newcommand{\calP}{\mathcal{P}}
\newcommand{\calS}{\mathcal{S}}
\newcommand{\cov}{\mathrm{cov}}
\newcommand{\NPB}{\mathrm{NPB}}
\newcommand{\ZZ}{\mathbb{Z}}
\newcommand{\e}{\mathbf{e}}
\newcommand{\ceil}[1]{\lceil#1\rceil}
\newcommand{\wceil}[1]{\left\lceil#1\right\rceil}

\newcommand{\gen}[1]{\left\langle#1\right\rangle}

\begin{document}
\title[PMD of graph products]{Positive matching decompositions of the Cartesian product of graphs}
\author[M. Farrokhi D. G.]{Mohammad Farrokhi D. G.}
\email{m.farrokhi.d.g@gmail.com,\ farrokhi@iasbs.ac.ir}
\address{Department of Mathematics, Institute for Advanced Studies in Basic Sciences (IASBS), and the Center for Research in Basic Sciences and Contemporary Technologies, IASBS, Zanjan 66731-45137, Iran}

\author[A. A. Yazdan Pour]{{Ali Akbar} {Yazdan Pour}}
\email{yazdan@iasbs.ac.ir}
\address{Department of Mathematics, Institute for Advanced Studies in Basic Sciences (IASBS), Zanjan 66731-45137, Iran}

	
\begin{abstract}
Let $\Gamma=(V,E)$ be a finite simple graph. A matching $M \subseteq E$ is positive if there exists a weight function on $V$ such that the matching $M$ is characterized by those edges with positive weights. A positive matching decomposition (pmd) of $\Gamma$ with $p$ parts is an ordered partition $E_1,\ldots,E_p$ of $E$ such that $E_i$ is a positive matching of $(V, E \setminus \bigcup_{j=1}^{i-1} E_j)$, for $i = 1, \ldots, p$. The smallest $p$ for which $\Gamma$ admits a pmd with $p$ parts is denoted by $\mathrm{pmd}(\Gamma)$. We study the pmd of the Cartesian product of graphs and give sharp upper bounds for them in terms of the pmds and chromatic numbers of their components. In special cases, we compute the pmd of grid graphs that is the Cartesian product of paths and cycles.
\end{abstract}

\maketitle
\section{Introduction and preliminaries}
Let $\G=(V,E)$ be a finite simple graph. A subset $M \subseteq E$ is a \textit{matching} in $\G$ if $e \cap e' =\varnothing$ for all $e, e' \in M$ with $e \neq e'$. A matching $M$ is called a \textit{positive matching} if there exists a weight function $w \colon V \to \mathbb{R}$ satisfying $w(uv):=w(u)+w(v)>0$ for an edge $uv\in E$ if and only if $uv\in M$.

A \textit{positive matching decomposition} (pmd) of $\G$ is an ordered partition $E_1,\ldots,E_p$ of $E$ such that $E_i$ is a positive matching of $(V, E \setminus \bigcup_{j=1}^{i-1} E_j)$, for $i = 1, \ldots, p$. The $E_i$'s are called the \textit{parts} of the pmd and the smallest $p$ for which $\G $ admits a pmd with $p$ parts is denoted by $\pmd(\G)$.

In what follows, ``the pmd'' $\pmd(\G)$ of a graph $\G$ denotes the minimum size of all positive matching decompositions of $\Gamma$ while ``a pmd'' of $\G$ refers to any positive matching decomposition of $\G$.

Positive matching decompositions of a graph $\G$ is introduced in \cite{ac-vw} where the authors study the algebraic properties of Lov\'{a}sz-Saks-Schrijver ideal (LSS-ideal) of the graph $\G$. Let $[n]=\{1, \ldots, n\}$. For a graph $\G=([n],E)$ the associated \textit{LSS-ideal} is the graded ideal
\[ L_{\G}^\mathbb{K}(d) = \left( x_{i,1}x_{j,1} + \cdots + x_{i,d}x_{j,d} \colon \quad \{i, j\} \in E\right),\]
in the polynomial ring $S=\mathbb{K}[x_{i,k} \colon \; i = 1,\ldots,n,\; k = 1, \ldots, d]$, where $\mathbb{K}$ is a field and $S$ is endowed with the standard grading. In this regard, it is shown that
\[d \geq \pmd(\G) \Rightarrow L_{\G}^\mathbb{K}(d) \text{ is radical and complete intersection} \Rightarrow L_{\G}^\mathbb{K}(d+1) \text{ is a prime ideal}.\]
In geometric setting, the vanishing locus $\mathcal{V}(L_{\bar{\G}}^\mathbb{R}(d)) \subseteq \mathbb{R}^{n\times d}$ is the set of all orthogonal representations of $\G$ in $\mathbb{R}^d$. We recall that an \textit{orthogonal representation} of $\G$ in $\mathbb{R}^d$ assigns to each $i \in [n]$ a vector $u_i \in \mathbb{R}^d$ such that $u^T_i u_j = 0$ whenever $\{i,j\} \in \bar{E}$. On the other hand, if $\mathbb{K}$ is an algebraically closed field and $\phi \colon \left(\mathbb{K}^n\right)^d \to \mathbb{K}^n \otimes \mathbb{K}^n$ is the map defined by 
\[(v_1, \ldots, v_d) \mapsto \sum_{j=1}^d v_j \otimes v_j\]
then the Zariski closure of the image of $\phi$ is the variety $S^d_{n,2}$ of symmetric tensors of (symmetric) rank $\leq d$ while its restriction to $\mathcal{V}(L_{\G}^\mathbb{K}(d))$ is a parameterization of the coordinate section of $S^d_{n,2}$ with $0$ coefficient at $e_i \otimes e_j$, for $\{i,j\} \in E(\G)$ (see \cite{sg-vw}). In particular, the Zariski closure of the image of the restriction is irreducible if $L_{\G}^\mathbb{K}(d)$ is prime which is the case when $d \geq \pmd(\G)+1$.

The pmd of graphs is studied in depth in \cite{mfdg-sg-aayp} where the authors give some general results on the pmd of graphs including upper and lower bounds as well as computing the pmd of some special graphs. Regarding the hypergraph setting, we may refer the reader to \cite{sg-vw, man-nk} for the pmd of particular classes of hypergraphs. 

Products of graphs yield important families of (complex) graphs used in the theory of networks, etc. Motivated by this, we aim to compute/evaluate the pmd of product of graphs in terms of the combinatorics of the given graphs. More explicitly, we give various sharp upper (and lower) bounds for the pmd of the Cartesian product of graphs.

Given two graphs $\G_1=(V_1,E_1)$ and $\G_2=(V_2,E_2)$, the \textit{Cartesian product} (or \textit{box product}) $\G=\G_1 \square \G_2$ of $\G_1$ and $\G_2$ is a graph on $V_1 \times V_2$ with edges $(a_1,a_2)\sim (b_1,b_2)$ if either $a_1=b_1$ and $a_2 \sim b_2$ or $a_1 \sim b_1$ and $a_2=b_2$. Some important graphs obtained from the Cartesian product of graphs are hypercubes $Q_n=K_2\square\cdots\square K_2$ ($n$ times), and grid graphs $P_m\square P_n$, $P_m\square C_n$, and $C_m\square C_n$ of the plane, the cylinder, and the torus, respectively.

In section 2, we give a couple of upper bounds for the pmd of the Cartesian product of two graphs. A simple immediate result shows that $\pmd(\G_1\square\G_2)$ is bounded above by 
\[\pmd(\G_1)\chi(\G_2)+\pmd(\G_2)\quad\text{and}\quad\pmd(\G_2)\chi(\G_1)+\pmd(\G_1)\]
(see Proposition \ref{pmd(G1OG2)<=min(pmd(G1)chi(G2)+pmd(G2) and pmd(G2)chi(G1)+pmd(G1))}). Here $\chi(\G)$ denotes the chromatic number of a graph $\G$. 

Let $F_1,\ldots,F_n$ be a forest decomposition of $\G_1$, and $F'_i$ be the subgraph of $\G_1-E(F_1)\cup\cdots\cup E(F_{i-1})$ induced on $V(F_i)$, for $i=1,\ldots,n$. Utilizing the given forest decomposition and the notion of (generalized) Latin rectangles, we get the following akin formulas
\begin{equation}\label{Forest decomposition bound 1}
\pmd(\G_1\square\G_2)\leq n\cdot\pmd(\G_2)+\sum_{i=1}^n\max\{\Delta(F'_i),\chi(\G_2)\}
\end{equation}
and
\begin{equation}\label{Forest decomposition bound 2}
\pmd(\G_1\square\G_2)\leq\pmd(\G_2)+n\cdot\max\left\{\sum_{i=1}^n\Delta(F'_i),\chi(\G_2)\right\},
\end{equation}
where $\sum_{i=1}^m\Delta(F'_i)\leq|\G_2|$ in the latter case (see Theorem \ref{pmd(G1 square G2): forest decomposition 1} and Theorem \ref{pmd(G1 square G2): forest decomposition 2}). Recall that $\Delta(\G)$ is the maximum degree of a graph $\G$.

In Section 3, we apply an stronger version of the inequality \eqref{Forest decomposition bound 2} to find an explicit upper bound for $\pmd(\G_1\square\G_2)$ when $\G_1$ is either a cycle graph or a non-prefix binary graph (including all complete multipartite graphs). It is shown that, in the case of non-prefix binary graphs, the upper bound for $\pmd(\G_1\square\G_2)$ relies on the problem of studying edge-coverings of $n*\G_1$ (the multigraph obtained from $\G_1$ whose edges are replaced with $n$ parallel edges) by forests. The minimum size of such a covering is guaranteed by the famous Nash-William's Theorem.

The last section of this paper concerns to study the pmd of the Cartesian product of paths and cycles. We give sharp lower and upper bounds for the pmd of such graphs by providing explicit positive matching decompositions.

Throughout this paper, we invoke the following theorem as a criteria to investigate if a given matching is positive.
\begin{theorem}[{\cite[Theorem 2.1]{mfdg-sg-aayp}}]\label{positivity of matchings}
Let $M$ be a matching in a graph $\G$. The following conditions are equivalent:
\begin{itemize}
\item[\rm (i)] $M$ is positive;
\item[\rm (ii)] The subgraph induced by $M$ does not contain any alternating closed walk;
\item[\rm (iii)] The subgraph induced by every subset of $M$ contains a pendant edge belonging to $M$;
\item[\rm (iv)] There exists an ordering of $M$ as $M=\{e_1,\ldots,e_n\}$ such that $e_i$ is a pendant edge in the subgraph induced by $\{e_1,\ldots,e_i\}$, for $i=1,\ldots,n$.
\end{itemize}
\end{theorem}

Recall that an \textit{alternating walk} in a graph $\G$ with respect to a matching $M$ is a walk whose edges alternate between edges in $M$ and $E(\G)\setminus M$. The subgraph of $\G$ induced by a set $E$ of edges is simply $\G[V(E)]$ by which we mean the subgraph of $\G$ induced by the vertex set of $E$. Also, the graph obtained from $\G$ by removing edges in $E$ is denoted by $\G-E$. Analogously, the graph obtained from $\G$ by removing a set $V$ of vertices of $\G$ along with incident edges is denoted by $\G-V$.
\section{Upper bounds for $\pmd(\G_1\square\G_2)$}
In this section, we apply various notions from graph theory and combinatorics (namely graph coloring, graph edge-covering, Latin rectangles, etc.) to give some (sharp) upper bounds for the pmd of the Cartesian products of two graphs. Moreover, we give explicit upper bounds in the case where the graphs under considerations are suitably chosen.	

In what follows, the product of a vertex $u$ of a graph $\G_1$ and an edge $e=vw$ of a graph $\G_2$ is defined to be the edge $u\times vw:=\{(u,v),(u,w)\}$ of $\G_1\square\G_2$. Accordingly, we define
\[U\times E:=\{u\times vw\colon u\in U, vw\in E\}\]
for any $U\subseteq V(\G_1)$ and $E\subseteq E(\G_2)$.

The following result yields a simple upper bound for $\pmd(\G_1\square\G_2)$ in terms of the pmd and the chromatic number of $\G_1$ and $\G_2$.
\begin{proposition}\label{pmd(G1OG2)<=min(pmd(G1)chi(G2)+pmd(G2) and pmd(G2)chi(G1)+pmd(G1))}
Let $\G_1$ and $\G_2$ be two graphs. Then
\[\pmd(\G_1\square\G_2)\leq\min\{\pmd(\G_1)\chi(\G_2)+\pmd(\G_2), \pmd(\G_2)\chi(\G_1)+\pmd(\G_1)\}.\]
\end{proposition}
\begin{proof}
Let $i\in\{1,2\}$ and $j\in\{1,2\}\setminus\{i\}$, $M_1^k,\ldots,M_{p_k}^k$ denote a PMD of $\G_k$ with $p_k=\pmd(\G_k)$, and $C_1^k,\ldots,C_{\chi_k}^k$ be the color classes of a coloring of $\G_k$ with $\chi_k=\chi(\G_k)$, for $k=1,2$. It is obvious, from Theorem \ref{positivity of matchings}, that 
\[M_1^i\times C_1^j,\ldots,M_{p_1}^i\times C_1^j,\ldots,M_1^i\times C_{\chi_j}^j,\ldots,M_{p_1}^i\times C_{\chi_j}^j,V(P_i)\times M_1^j,\ldots,V(P_i)\times M_{p_2}^j\]
is a pmd of $\G_1\square\G_2$, from which the result follows.
\end{proof}
\begin{corollary}\label{Corollary to pmd(G1OG2)<=min(pmd(Gi)chi(Gj)+pmd(Gj)}
Let $\G,\G_1,\ldots,\G_n$ be graphs with $n\geq1$. Then
\begin{itemize}
\item[(i)]$\pmd(\G\square Q_n)\leq\pmd(\G)+n\cdot\chi(\G)\leq (n+1)\cdot\pmd(\G)$ except for $\G=K_2$. 
\item[(ii)]if $\pmd(\G_1)=\cdots=\pmd(\G_n)$ and $\chi(\G_1)\leq\cdots\leq\chi(\G_n)$, then
\[\pmd(\G_1\square\cdots\square\G_n)\leq(\chi(\G_1)+\cdots+\chi(\G_{n-1})+1)\cdot\pmd(\G_n).\]
\item[(iii)]$\pmd(\square^n\G)\leq((n-1)\chi(\G)+1)\cdot\pmd(\G)$. 
\item[(iv)]$\pmd(Q_n)\leq2n-1$.
\end{itemize}
\end{corollary}
\begin{proof}
(i) The results follows from  \cref{pmd(G1OG2)<=min(pmd(G1)chi(G2)+pmd(G2) and pmd(G2)chi(G1)+pmd(G1))}, Brooks' Theorem, and induction on $n$.

(ii) The result follows from  \cref{pmd(G1OG2)<=min(pmd(G1)chi(G2)+pmd(G2) and pmd(G2)chi(G1)+pmd(G1))}, \cite[Theorem 26.1]{rh-wi-sk}, and induction on $n$.

(iii) It follows from (ii). 

(iv) follows from (iii) assuming that $\G=K_2$.
\end{proof}
\begin{remark}
The upper bound in Proposition \ref{pmd(G1OG2)<=min(pmd(G1)chi(G2)+pmd(G2) and pmd(G2)chi(G1)+pmd(G1))} is sharp and the equality occurs for many graphs. However, the pmd of $\G_1\square\G_2$ need not to attain the upper bound in general. For instance, $\pmd(\G_1\square\G_2)=3$ for $(\G_1,\G_2)=(P_m,K_2)$ while the upper bound is $4$ (see \eqref{pmd(T1O...OTnOK2): Ti<>K2}).
\end{remark}
\begin{question}
Let $\G$ be a graph. Does there exist a closed formula for $\pmd(\G\square K_2)-\pmd(\G)$?
\end{question}

In the rest of this section, we invoke forest decompositions of graphs to evaluate the pmd of the Cartesian product of graphs. Our first result reduces the problem of bounding $\pmd(\G_1\square\G_2)$ to that of $\pmd(\G_1\square F_2)$, where $F_2$ is any forest in a forest decomposition of $\G_2$.
\begin{definition}
A \textit{forest decomposition} of a graph $\G$ is an ordered edge-partition of $\G$ into forests $F_1,\ldots,F_n$ such that the subgraph of $\G-E(F_1)\cup\cdots\cup E(F_{i-1})$ induced by $V(F_i)$ is a forest, for all $i=1,\ldots,n$.
\end{definition}
\begin{theorem}\label{pmd(G1 square G2): forest decomposition 1}
Let $\G_1$ and $\G_2$ be two graphs, and $F_1,\ldots,F_n$ be a forest decomposition of $\G_1$. If $F'_i:=(\G_1-E(F_1)\cup\cdots\cup E(F_{i-1}))[V(F_i)]$ for $i=1,\ldots,n$, then 
\[\pmd(\G_1\square\G_2)\leq n\cdot\pmd(\G_2)+\sum_{i=1}^n\max\{\Delta(F'_i),\chi(\G_2)\}.\]
\end{theorem}
\begin{proof}
Let $C_1,\ldots,C_\chi$ be coloring classes of $\G_2$ with $\chi=\chi(\G_2)$. Let $T$ be any tree and $M_1,\ldots,M_\Delta$ be a minimum sized pmd of $T$ with $\Delta=\Delta(T)$ (see \cite[Lemma 5.4]{ac-vw} or \cite[Corollary 2.4]{mfdg-sg-aayp}). First we show that 
\[\pmd(T\square\G_2)\leq\pmd(\G_2)+\max(\Delta(T), \chi(\G_2)).\]
Let $L=[L_{ij}]$ be any $\Delta\times\chi$ Latin rectangle with entries $1,\ldots,m$ and $m=\max(\Delta,\chi)$. Let
\[M'_i:=\bigcup_{L_{st}=i}M_s\times' C_t,\]
for $i=1,\ldots,m$, where $M_s\times' C_t$ is the set of all edges $ab\times u=\{(a,u),(b,u)\}$ with $ab\in M_s$ and $u\in C_t$. We show that $M'_1,\ldots,M'_m$ are positive matchings of $T\square\G_2$ covering the set of all vertical edges $E(T)\times' V(\G_2)$. To this end, let $u_1v_1,\ldots,u_qv_q$ be the edges of $T$ ordered in such a way that $u_i$ is pendant in $T-\{u_1v_1,\ldots,u_{i-1}v_{i-1}\}$, for $i=1,\ldots,q$. Since $u_i$ is a pendant in $T-\{u_1v_1,\ldots,u_{i-1}v_{i-1}\}$, it follows that for every $1\leq i\leq\Delta$ and $1\leq j\leq q$, either $(u_j\times C_t)\cap M_i=\emptyset$, or $u_j\times C_t$ is a set of pendants in 
\[M'_i-\bigcup_{\substack{k<j\\u_kv_k\in M_{s'}\\L_{s't'}=i}}u_kv_k\times C_{t'},\]
where $t$ is such that $L_{st}=i$ provided that $u_jv_j\in M_s$. Hence, by \cref{positivity of matchings}(iv), $M'_i$ is a positive matching in $T\square\G_2$, for $i=1,\ldots,m$. Since $T\square\G_2-M'_1-\cdots-M'_m=|T|\G_2$ is the union of disjoint copies of $\G_2$, it follows that $\pmd(T\square\G_2)\leq m+\pmd(\G_2)$, as required.

Now, from the definition of $F'_1,\ldots,F'_n$, we observe that
\begin{align*}
\pmd(\G_1\square\G_2)&\leq\pmd(F'_1\square\G_2)+\cdots+\pmd(F'_n\square\G_n)\\
&=n\cdot\pmd(\G_2)+\sum_{i=1}^n\max\{\Delta(F'_i),\chi\}.
\end{align*}
The proof is complete.
\end{proof}
\begin{corollary}\label{pmd(T1 square ... square Tn)}
If $T_1,\ldots,T_n$ are trees, then
\[\pmd(T_1\square\cdots\square T_n)\leq\Delta(T_1\square\cdots\square T_n)+m-\delta_{m\neq0},\]
where $m:=\#\{i\colon T_i\cong K_2\}$.
In particular,
\begin{align}
\pmd(T_1\square\cdots\square T_n)&=\Delta(T_1\square\cdots\square T_n),\label{pmd(T1O...OTn): Ti<>K2}\\
\pmd(T_1\square\cdots\square T_n\square K_2)&=\Delta(T_1\square\cdots\square T_n\square K_2),\label{pmd(T1O...OTnOK2): Ti<>K2}
\end{align}
if $T_1,\ldots,T_n\ncong K_2$.
\end{corollary}
\begin{proof}
We show that
\[\pmd(T_1\square\cdots\square T_n\square K_2^{\square m})\leq\Delta(T_1\square\cdots\square T_n)+2m-\delta_{m\neq0}\]
where $m,n\geq0$ and $T_1,\ldots,T_n\ncong K_2$ are trees.

If $n=0$, then the result follows from Corollary \ref{Corollary to pmd(G1OG2)<=min(pmd(Gi)chi(Gj)+pmd(Gj)}(iv). Hence, assume that $n\geq1$. Clearly, the result holds for $(m,n)=(0,1)$ by \cite[Corollary 2.4]{mfdg-sg-aayp}. Suppose the result holds for $n-1$ and any $m\geq0$ with $(m,n)\neq(0,1)$. By  \cref{pmd(G1 square G2): forest decomposition 1}, we have
\begin{multline*}
\pmd(T_1\square T_2\square\cdots\square T_n\square K_2^{\square m})\leq\\
\pmd(T_2\square\cdots\square T_n\square K_2^{\square m})+\max(\Delta(T_1),\chi(T_2\square\cdots\square T_n\square K_2^{\square m})).
\end{multline*}
Since $\chi(T_2\square\cdots\square T_n\square K_2^{\square m})=2$ by Sabidussi's Theorem (see \cite[Theorem 26.1]{rh-wi-sk}), we observe that
\[\pmd(T_1\square T_2\square\cdots\square T_n\square K_2^{\square m})\leq\Delta(T_1)+\Delta(T_2)+\cdots+\Delta(T_n)+\pmd(K_2^{\square m})\]
by induction hypothesis. Thus
\[\pmd(T_1\square T_2\square\cdots\square T_n\square K_2^{\square m})\leq\Delta(T_1)+\Delta(T_2)+\cdots+\Delta(T_n)+2m-1\]
by Corollary \ref{Corollary to pmd(G1OG2)<=min(pmd(Gi)chi(Gj)+pmd(Gj)}(iv). 

Since $\pmd(T_1\square \cdots\square T_n\square K_2^\varepsilon)\geq\Delta(T_1)+\cdots+\Delta(T_n)+\varepsilon$ ($\varepsilon=0,1$), it follows that 
\[\pmd(T_1\square \cdots\square T_n\square K_2^\varepsilon)=\Delta(T_1)+\cdots+\Delta(T_n)+\varepsilon,\]
as required.
\end{proof}

The above corollary suggest us to pose the following conjecture.
\begin{conjecture}
If $T_1,\ldots,T_n$ are trees, then
\[\pmd(T_1\square\cdots\square T_n)=\Delta(T_1\square\cdots\square T_n)+m-\delta_{m\neq0},\]
where $m:=\#\{i\colon T_i\cong K_2\}\leq3$.
\end{conjecture}
\begin{remark}
The above conjecture is not true for $m\geq4$. Indeed, we have $\pmd(K_2^{\square 4})=6$ (see Proposition \ref{pmd(C3OC3)=pmd(C4OC4)=6}) while the upper bound in Corollary \ref{pmd(T1 square ... square Tn)} is $7$. This gives a negative answer to our previous conjecture (see \cite[Conjecture 4.5]{mfdg-sg-aayp}) so a possible closed formula for $\pmd(Q_n)$ remains unknown.
\end{remark}

Theorem \ref{pmd(G1 square G2): forest decomposition 1} has a twisted version described in Theorem \ref{pmd(G1 square G2): forest decomposition 2}. To achieve this version, we first obtain a result in a more general and intricate setting as in Theorem \ref{pmd(G_1 O G_2): AFE-cover}. This theorem relies on a couple of notions described below.
\begin{definition}
Let $\calS$ be a multiset, $\calF=\{\calS_1,\ldots,\calS_k\}$ be a family of submultisets of $\calS$, and $\calP$ be a set-theoretical property on subsets of $2^\calS$. A \textit{$\calP$-cover} of $\calS$ by $\calF$ is a family $\calC=\{\calC_i\}_{i\in I}$ of subsets of $\calF$ satisfying 
\begin{itemize}
\item[(1)]$\calS=\cup_{i\in I}(\cup\calC_i)$,
\item[(2)]$\calC$ is a $\calP$-subset of $2^\calS$.
\end{itemize}
The minimum size of a $\calP$-cover of $\calS$ by $\calF$ is called the \textit{$\calP$-covering number} of $\calS$ by $\calF$ denoted by $\cov_\calP(\calS, \calF)$. If $\calS=E(\G)$ for a multigraph $\G$, we usually write $\cov_\calP(\G, \calF)$ for $\cov_\calP(\calS, \calF)$.
\end{definition}
\begin{definition}
Let $\G$ be a graph. A family $\calE$ of subsets of $E(\G)$ is called an \textit{acyclic family of edge-sets} (AFE) if the graph $\cup_{E\in\calE}\G[E]$ is acyclic.
\end{definition}
\begin{definition}
A \textit{generalized $m\times n$ Latin rectangle} is an $m\times n$ array of numbers with no number occurring more than once in any row or column.
\end{definition}

In what follows, $n*X$ stands for the multiset $\{x^n\colon x\in X\}$, where $x^n$ stands for $n$ copies of $x$ for all elements $x$ of a set $X$. Analogously, $n*\G$ stands for the multigraph $(V(\G),n*E(\G))$ for any graph $\G$.
\begin{lemma}\label{Existence of generalized m x n Latin rectangle}
Let $n*[m]=X_1\cup\cdots\cup X_k$ be a partition of the multiset $n*[m]$ into subsets, where $m\leq n$. Then there exists a generalized $m\times n$ Latin rectangle filled with $1,\ldots,k$ in such a way that $X_i$ is the set of rows the number $i$ appears for all $i=1,\ldots,k$.
\end{lemma}
\begin{proof}
We proceed by induction on $m$. Clearly, the result holds for $m=1$. Suppose $m>1$ and the result holds for $m-1$. Suppose without loss of generality that $m\in X_1,\ldots,X_n$ and $L$ is a generalized $(m-1)\times n$ Latin rectangle assosiated to the partition \[n*[m-1]=(X_1\setminus\{m\})\cup\cdots\cup(X_n\setminus\{m\})\cup X_{n+1}\cup\cdots\cup X_k\]
filled with $1,\ldots,k'$, where $k'=n-\#\{i\colon X_i=\{m\}\}$. Let $L'=[\ell'_{ij}]$ be the partial $(m-1)\times n$ Latin rectangle obtained from $L$ by erasing all entries greater than $n$. By \cite[Theorem 10.4.12]{rb}, there is a completion $L''$ of $L'$ as a Latin rectangle with entries $1,\ldots,n$. By \cite[Theorem 10.4.11]{rb}, $L''=[\ell''_{ij}]$ can be extended to an $m\times n$ Latin rectangle $L'''=[\ell'''_{ij}]$. Finally, we define the generalized $m\times n$ Lating rectangle $L''''=[\ell''''_{ij}]$ as follows:
\[\ell''''_{ij}=\begin{cases}
\ell'''_{ij},&\ell_{ij}=\ell''_{ij},\\
\ell_{ij},&\ell_{ij}\neq\ell''_{ij}.
\end{cases}\]
Then $L''''$ satisfies the required properties.
\end{proof}

Utilizing the above arguments, we are now in the position to prove the core theorem of this paper (in some sense) from which we conclude the twisted version of Theorem \ref{pmd(G1 square G2): forest decomposition 1} immediately.
\begin{theorem}\label{pmd(G_1 O G_2): AFE-cover}
Let $\G_1$ and $\G_2$ be two graphs. Let $\calM_i:M_1^i,\ldots,M_{p_i}^i$ be a pmd of $\G_i$ and $C_1^i,\ldots,C_{\chi_i}^i$ be color classes of a proper coloring of $\G_i$ for $i=1,2$. Then
\[\pmd(\G_1\square\G_2)\leq\inf\{\cov_{\mathrm{AFE}}(\chi_j*\G_i,\calM_i)+p_j\colon \{i,j\}=\{1,2\}\text{ and }p_i\leq\chi_j\}.\]
\end{theorem}
\begin{proof}
Let $\{\calC_k\}_{k\in I}$ be an AFE-cover of $\chi_2*E(\G_1)$ by $\calM_1$ of minimum size, where $\calC_k=\cup_{h\in I_k}M_h^{1,k}$ and $I_k\subseteq[p_1]$ for all $k\in I$. Notice that $M_h^{1,k}=M_h^1$ as an element of $\calM_1$ for all $h\in I_k^1$. Assume $p_1\leq\chi_2$ and $L=[\ell_{ij}]$ is a generalized $p_1\times\chi_2$ Latin rectangle associated to the partition $\{\calC_k\}_{k\in I}$ of $\chi_2*E(\G_1)$ (see \cref{Existence of generalized m x n Latin rectangle}). Viewing $\chi_2*E(\G_1)$ as $E(\chi_2\G_1)=E(\G_1^1\sqcup\cdots\sqcup\G_1^{\chi_2})$, where $\G_1^1,\ldots, \G_1^{\chi_2}\cong\G_1$, we may assume without loss of generality that $M_h^{1,k}\subseteq E(\G_1^t)$ for all $h\in I_k$ and $k\in I$, where $t$ is such that $\ell_{h,t}=k$. We claim that
\[U_k:=\bigcup_{h\in I_k,\ \ell_{h,t}=k}M_h^{1,k}\times C_t^2\]
is a positive matching of $\G_1\square\G_2$ for all $k\in I$. Clearly, $U_k$ is a matching of $\G_1\square\G_2$ as 
\[V(M_h^{1,k}\times C_t^2)\cap V(M_{h'}^{1,k}\times C_{t'}^2)=\varnothing\]
for all $(h,t)\neq(h',t')$ with $h,h'\in I_k$ and $\ell_{h,t}=\ell_{h',t'}=k$. If $(\G_1\square\G_2)[U_k]$ has an alternating closed walk, then the vertical edges $uv\times w$ yield a closed walk in $\cup_{h\in I_k}\G[M_h^{1,k}]$ possessing of the corresponding edges $uv$, contradicting the assumption. Thus $U_k$ is a positive matching of $\G_1\square\G_2$. Since $\{M_1^2,\ldots,M_{p_2}^2\}$ is a pmd of $\G_2$ and $I=[p'_1]$, it follows that
\[U_1,\ldots,U_{p'_1},V(\G_1)\times M_1^2,\ldots,V(\G_1)\times M_{p_2}^2\]
is a pmd of $\G_1\square\G_2$. Therefore,
\[\pmd(\G_1\square\G_2)\leq p'_1+p_2=\cov_{\mathrm{AFE}}(\chi_2*\G_1,\calM_1)+p_2.\]
A similar argument yields 
\[\pmd(\G_1\square\G_2)\leq\cov_{\mathrm{AFE}}(\chi_1*\G_2,\calM_2)+p_1\]
whenever $p_2\leq\chi_1$, from which the result follows.
\end{proof}
\begin{theorem}\label{pmd(G1 square G2): forest decomposition 2}
Let $\G_1$ and $\G_2$ be two graphs, and $F_1,\ldots,F_n$ be a forest decomposition of $\G_1$. If $F'_i:=(\G_1-E(F_1)\cup\cdots\cup E(F_{i-1}))[V(F_i)]$ for $i=1,\ldots,n$, then 
\[\pmd(\G_1\square\G_2)\leq\pmd(\G_2)+n\cdot\max\left\{\sum_{i=1}^n\Delta(F'_i),\chi(\G_2)\right\}\]
provided that $\sum_{i=1}^m\Delta(F'_i)\leq|\G_2|$.
\end{theorem}
\begin{proof}
Let $\chi_2:=\max\{\sum_{i=1}^n\Delta(F'_i),\chi(\G_2)\}$. If $\calM_i:M_{i,1},\ldots,M_{i,\Delta(F'_i)}$ is a pmd of $F'_i$ for $i=1,\ldots,n$, then $\calM:\calM_1,\ldots,\calM_n$ is a pmd of $\G_1$ and $\chi_2*\{\calM_1,\ldots,\calM_n\}$ is an AFE-cover of $\chi_2*E(\G_1)$ by $\calM$. As $p_1:=|\calM|\leq\chi_2$, it follows from \cref{pmd(G_1 O G_2): AFE-cover} that
\[\pmd(\G_1\square\G_2)\leq n\cdot\chi_2+\pmd(\G_2),\]
as required.
\end{proof}
\begin{corollary}
Let $T$ be a tree. If $\G$ is a graph satisfying $\Delta(T)\leq|\G|$, then
\[\pmd(\G\square T)\leq\pmd(\G)+\max\{\Delta(T), \chi(\G)\}.\]
\end{corollary}

All results we had so far focus on upper bounds for the pmd of the Cartesian product of two graphs. We have no significant result for the lower bound except for the maximum valency. However, we believe that the pmd of the Cartesian product of two graphs, scaled up to a constant value, is bounded below by sum of the pmd of its components.
\begin{conjecture}
There exists a constant $c>0$ such that 
\[\pmd(\G_1)+\pmd(\G_2)\leq\pmd(\G_1\square\G_2)+c\]
for all graphs $\G_1$ and $\G_2$.
\end{conjecture}

The constant $c$ above should indeed be positive (see Theorem \ref{pmd(Cm O Cn)} with $(m,n)=(4,5)$).
\section{Acyclic family of edge-sets relative to a pmd}
We know from Theorem \ref{pmd(G_1 O G_2): AFE-cover} that 
\[\pmd(\G_1\square\G_2)\leq\inf\{\cov_{\mathrm{AFE}}(\chi_j*\G_i,\calM_i)+p_j\colon \{i,j\}=\{1,2\}\text{ and }p_i\leq\chi_j\}\]
for any two graphs $\G_1$ and $\G_2$, where $\calM_i:M_1^i,\ldots,M_{p_i}^i$ is a pmd of $\G_i$ and $C_1^i,\ldots,C_{\chi_i}^i$ are color classes of a proper coloring of $\G_i$ for $i=1,2$. The aim of this section is to study the quantities $\cov_{\mathrm{AFE}}(\chi_j*\G_i,\calM_i)$ above.

A \textit{clutter} $\calC$ on a vertex set $V$ is an anti-chain in the lattice of all subsets of $V$ provided that $\cup\calC=V$. The elements of $\calC$ are known as the \textit{circuits} of $\calC$. A \textit{vertex $n$-cover} of the clutter $\calC$ is a family of circuits of $\calC$ that covers every vertex at least $n$ times. The minimum size of a vertex $n$-cover of $\calC$ is denoted by $\tau_n(\calC)$. Assume $\calS=n*[m]$ and $\calF$ is a partition of $[m]$. If $\calC_\calP(\calS, \calF)$ is the clutter on $\calF$ with all maximal $\calP$-subsets of $2^\calS$ as circuits, then 
\[\cov_\calP(\calS, \calF)=\tau_n(\calC_\calP(\calS, \calF)).\]
Based on this observation, in order to compute the upper bound in Theorem \ref{pmd(G_1 O G_2): AFE-cover}, we need to compute the following invariant
\begin{equation}\label{kappa definition}
\kappa_{\G}(n, p):=\inf\{\tau_n(\calC_{\mathrm{AFE}}(n*\G,\calM))\colon \calM\text{ is a pmd of $\G$ of size }p\}
\end{equation}
for all $p\geq1$. It turns out that
\begin{equation}\label{pmd(G O G') <= k_G(chi, p) + p}
\pmd(\G\square\G')\leq\kappa_{\G}(\chi, p)+p'
\end{equation}
for any $\chi$-coloring of $\G'$ with $\chi\geq p\geq\pmd(\G)$ and $p'=\pmd(\G')$.

In what follows, we calculate the quantity $\kappa_\G(n, p)$ for some classes of graphs including cycles and complete multipartite graphs.
\begin{lemma}\label{Cover of n*[m] by k-subsets}
The minimum number of $k$-subsets of $[m]$ to cover $n*[m]$ is $\lceil nm/k\rceil$.
\end{lemma}
\begin{proof}
If $C_1,\ldots,C_i$ is a cover of $n*[m]$ into $k$-subsets of $[m]$, then $ik=|C_1|+\cdots+|C_i|\geq mn$ so that $i\geq\lceil mn/k\rceil$. We construct a cover with exactly $\lceil mn/k\rceil$, from which the result follows.

Let $\{a_i\}$ be the sequence of numbers obtained by concatenating $1,\ldots,m$ infinitely often. Let $C_i=\{a_{(i-1)k+1},\ldots,a_{ik}\}$, for all $i\geq1$. Then $C_1,\ldots,C_i$ covers $n*[m]$ if and only if $ik\geq mn$ or equivalently $i\geq\lceil nm/k\rceil$. 
\end{proof}
\begin{proposition}
If $\G=C_m$ is a cyclic graph and $m\geq p\geq3$ with $(m,p)\neq(4,3)$, then
\[\kappa_{\G}(n, p)=\left\lceil\frac{pn}{p-1}\right\rceil\]
except when $(m,p)=(5,4)$ for which $\kappa_{\G}(n, p)=\lceil\frac{3n}{2}\rceil$.
\end{proposition}
\begin{proof}
Since any member of an acyclic family of edge-sets relative to a pmd $M_1,\ldots,M_p$ of size $p$ of $\G=C_m$ possesses at most $p-1$ of $M_i$'s, we observe that $\kappa_{\G}(n, p)\geq pn/(p-1)$. We show that $\kappa_{\G}(n, p)$ takes the lower bound except for $(m,p)=(5,4)$. Let $e_1,\ldots,e_m$ be the consecutive edges of $\G$.

First assume that $(m,p)\neq(5,4)$. We show that $\G$ has a pmd $\calM:M_1,\ldots,M_p$ such that 
\begin{equation}\label{M_1,...,^M_i,...,M_p is acyclic}
	\bigcup_{j=i}^{i+p-1}E(\G[M_j])\ \text{is acyclic for all}\ i=1,\ldots,p, 
\end{equation}
where the index $j$ is taken modulo $p$, namely $M_{p+1}=M_1$, $M_{p+2}=M_2$, etc. If $m=p$, then simply take $\calM:\{e_1\},\ldots,\{e_m\}$. Hence, assume that $m>p$. Then $m\geq5$ by assumption. If $m=5$, then $p=3$ and we may put $\calM:\{e_1\},\{e_2,e_4\},\{e_3,e_5\}$. Now, let $m\geq6$. Put $\calM:\{e_1\},\{e_2\},\{e_4\},\{e_5\},\ldots,\{e_{m-1}\},\{e_3,e_m\}$ if $m=p+1$. For $m\geq p+2$, we define $\calM:M_1,\ldots,M_p$ as follows:
\begin{align*}
M_1&=\begin{cases}
	\{e_1,e_{p+1},e_{p+3},\ldots,e_{m-3},e_{m-1}\},&m-p\ \text{is even},\\
	\{e_1,e_{p+1},e_{p+3},\ldots,e_{m-4},e_{m-2}\},&m-p\ \text{is odd},
\end{cases}\\
M_2&=\begin{cases}
	\{e_2,e_{p+2},e_{p+4},\ldots,e_{m-2},e_m\},&m-p\ \text{is even},\\
	\{e_2,e_{p+1},e_{p+3},\ldots,e_{m-3},e_{m-1}\},&m-p\ \text{is odd},
\end{cases}\\
M_3&=\begin{cases}
	\{e_3\},&m-p\ \text{is even},\\
	\{e_3,e_m\},&m-p\ \text{is odd},
\end{cases}	
\end{align*}
and $M_i=\{e_i\}$, for $i=4,\ldots,p$. A simple verification shows that $\calM$ satisfies \eqref{M_1,...,^M_i,...,M_p is acyclic}. Hence, by Lemma \ref{Cover of n*[m] by k-subsets}, $\kappa_{\G}(n, p)=\lceil\frac{pn}{p-1}\rceil$. Finally, let $(m,p)=(5,4)$. Without loss of generality, we may assume that $\calM:\{e_1,e_3\},\{e_2\}, \{e_4\}, \{e_5\}$ is a pmd of $\G$ with $4$ parts. Since 
\begin{align*}
\calC_{\mathrm{AFE}}(n*\G,\calM)&:=\begin{cases}
	\{\calC_1 (k\ \text{times}), \calC_2 (k\ \text{times}), \calC_3 (k\ \text{times})\},&n=2k,\\
	\{\calC_1 (k\ \text{times}), \calC_2 (k+1\ \text{times}), \calC_3 (k+1\ \text{times})\},&n=2k+1,
\end{cases}
\end{align*}
where $\calC_1=M_1\cup M_2\cup M_3$, $\calC_2=M_1\cup M_2\cup M_4$, and $\calC_3=M_3\cup M_4$. Then $\#\calC_{\mathrm{AFE}}(n*\G,\calM)=\lceil\frac{3n}{2}\rceil$. Since $\calM$ is unique up to symmetry, one can easily see that the set $\calC_{\mathrm{AFE}}(n*\G,\calM)$ above has minimum size among all possible acyclic families of edge-sets relative to $\calM$. Therefore, $\kappa_{\G}(n, p)=\lceil\frac{3n}{2}\rceil$, as required.
\end{proof}

Let $K_m$ and $K_{a,b}$ stand for complete and complete bipartite graphs, and $\G=K_m$ or $K_{a,b}$. It is evident that $n*E(\G)$ has no AFE-covering by any pmd of $\G$ containing a matching with at least two edges. Hence, the only pmd of $\G$ that yields an AFE-covering of $n*E(\G)$ is $\{\{e\}\colon e\in E(\G)\}$. In what follows, we classify all graphs $\G$ satisfying the same property that is the only pmd of $\G$ that yields an AFE-covering of $n*E(\G)$ is $\{\{e\}\colon e\in E(\G)\}$. Having the classification of such graphs $\G$, it enables us to give an explicit formula for $\kappa_{\G}\left(n,|E(\G)|\right)$.
\begin{definition}
A \textit{matching decomposition} (MD) $\calM=\{M_1,\ldots,M_p\}$ of a graph $\G$ is an edge decomposition of the edge set of $\G$ whose elements $M_1,\ldots,M_p$ are matchings. The matching decomposition $\calM$ is \textit{acyclic} (AMD) if $\G[M_i]$ is acyclic for all $i=1,\ldots,p$. Also, a matching decomposition is \textit{trivial} (TMD) if it possesses of only edges as matchings.
\end{definition}
\begin{definition}
A binary sequence $s'$ is a \textit{prefix} of a binary sequence $s$ if $s=s's''$ for some binary sequence $s''$. Let $n\geq1$ and $\Sigma_n:=\{s\in\{0,1\}^{[k]},\ k\in[n]\}$ be the set of all binary sequences $s$ of lengths $l(s)$ with $1\leq l(s)\leq n$. Also, let $\{A_s\}_{s\in\Sigma_n}$ be a family of disjoint sets such that $A_s=\varnothing$ if and only if $A_{s'}=\varnothing$ for $s,s'$ being the binary sequences of length $n$ and the same prefix of length $n-1$. Let $\G=\NPB(\{A_s\}_{s\in\Sigma_n})$ be the graph with vertex set $\cup A_s$ and edges $uv$ satisfying either
\begin{itemize}
\item $u\in A_s$, $v\in A_{s'}$, $s\neq s'$, and $l(s)=l(s')$, or
\item $u\in A_s$, $v\in A_{s'}$, $l(s)<l(s')$, and $s$ is not a prefix of $s'$.
\end{itemize}
Then $\G$ is called the \textit{non-prefix binary graph} with respect to $\{A_s\}_{s\in\Sigma_n}$.
\end{definition}
\begin{figure}[h]
\centering
\begin{tikzpicture}[scale=1]
\node [draw, circle, fill=white, inner sep=1pt] (0a) at (-1.5,0) {};

\node [draw, circle, fill=white, inner sep=1pt] (1a) at (1.0,0) {};
\node [draw, circle, fill=white, inner sep=1pt] (1b) at (1.5,0) {};
\node [draw, circle, fill=white, inner sep=1pt] (1c) at (2.0,0) {};

\node [draw, circle, fill=white, inner sep=1pt] (00a) at (-2.5,2) {};
\node [draw, circle, fill=white, inner sep=1pt] (01a) at (-1.5,2) {};
\node [draw, circle, fill=white, inner sep=1pt] (01b) at (-1.0,2) {};

\node [draw, circle, fill=white, inner sep=1pt] (10a) at (1,2) {};
\node [draw, circle, fill=white, inner sep=1pt] (11a) at (2.5,2) {};

\draw (10a)--(0a);
\draw (0a) to [out=20,,in=230] (11a);

\draw (0a) to [out=-30,in=-150] (1a);
\draw (0a) to [out=-30,in=-150] (1b);
\draw (0a) to [out=-30,in=-150] (1c);

\draw (00a) to [out=30,in=150] (01a);
\draw (00a) to [out=30,in=150] (01b);
\draw (00a) to [out=30,in=150] (10a);
\draw (00a) to [out=30,in=150] (11a);

\draw (01a) to [out=30,in=150] (10a);
\draw (01a) to [out=30,in=150] (11a);

\draw (01b) to [out=30,in=150] (10a);
\draw (01b) to [out=30,in=150] (11a);

\draw (10a) to [out=30,in=150] (11a);

\draw [dashed] (-1.75,-0.25) rectangle (-1.25,0.25);
\draw [dashed] (0.75,-0.25) rectangle (2.25,0.25);
\draw [dashed] (-2.75,1.75) rectangle (-2.25,2.25);
\draw [dashed] (-1.75,1.75) rectangle (-0.75,2.25);
\draw [dashed] (0.75,1.75) rectangle (1.25,2.25);
\draw [dashed] (2.25,1.75) rectangle (2.75,2.25);

\node [label={[label distance=0.25cm]left:$A_0$}] () at (-1.5,0) {};
\node [label={[label distance=0.25cm]right:$A_1$}] () at (2.0,0) {};

\node [label={[label distance=0.25cm]left:$A_{00}$}] () at (-2.5,2) {};
\node [label={[label distance=0.25cm]below:$A_{01}$}] () at (-1.25,2) {};

\node [label={[label distance=0.25cm]below:$A_{10}$}] () at (1,2) {};
\node [label={[label distance=0.25cm]right:$A_{11}$}] () at (2.5,2) {};
\end{tikzpicture}
\caption{A non-prefix binary graph with $|A_0|=|A_{00}|=|A_{10}|=|A_{11}|=1$, $|A_{01}|=2$, and $|A_1|=3$}
\label{A non-prefix binary graph with |A0|=|A00|=|A10|=|A11|=1, |A01|=2, and |A1|=3}
\end{figure}
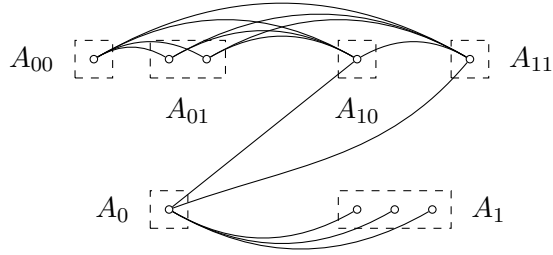

In what follows, we use the notion of join of graphs frequently. Recall that the join $\G_1*\G_2$ of two disjoint graphs $\G_1$ and $\G_2$ is the graph obtained from $\G_1$ and $\G_2$ by joining every vertex of $\G_1$ to every vertex of $\G_2$.
\begin{proposition}
Let $\G$ be a connected graph. Then the following conditions are equivalent:
\begin{itemize}
\item[(1)]The only AMD of $\G$ is the TMD of $\G$,
\item[(2)]Every two disjoint edges of $\G$ induce a cycle in $\G$,
\item[(3)]$\G$ is a non-prefix binary graph.
\end{itemize}
\end{proposition}
\begin{proof}
(1)$\Leftrightarrow$(2) It is straightforward.

(2)$\Rightarrow$(3) First observe that if $u,v\in V(\G)$ are not adjacent, then either $N_\G(u)\subseteq N_\G(v)$ or $N_\G(v)\subseteq N_\G(u)$ otherwise we have two disjoint edges $uu'$ and $vv'$ ($u'\in N_\G(u)\setminus N_\G(v)$ and $v'\in N_\G(v)\setminus N_\G(u)$) such that $\G[u,u',v,v']$ is acyclic.

Let $u\in V(\G)$ be such that $N_\G(u)$ is minimal, and put $\G_0:=\G[N_\G(u)]$ and $\G_1:=\G[V(\G)\setminus N_\G(u)]$. Since $N_\G(u)\subseteq N_\G(v)$ for all $v\notin N_\G(u)$, it follows that every vertex of $\G_0$ is adjacent to every vertex of $\G_1$ that is $\G\cong\G_0*\G_1$ is the join of $\G_0$ and $\G_1$. Since $\G_0$ and $\G_1$ satisfy (2), it follows that $\G_0=A_0\cup\G'_0$ and $\G_1=A_1\cup\G'_1$ for some connected graphs $\G'_0$ and $\G'_1$, and independent sets $A_0,A_1$ of $\G_0,\G_1$, respectively. Since $\G'_0$ and $\G'_1$ satisfy (2), an inductive argument shows that $\G'_0=\NPB(\{B_s\}_{s\in\Sigma_m})$ and $\G'_1=\NPB(\{C_s\}_{s\in\Sigma_n})$ for some $m\leq n$. Without loss of generality, we may assume that $m=n$ as we may set $B_s$ to be the empty set for all $s$ of lengths $m+1,\ldots,n$. Let 
\[A_s:=\begin{cases}
B_{s'},&s=0s',\\
C_{s'},&s=1s'
\end{cases}\]
for all $s\in\Sigma_{n+1}\setminus\{0,1\}$. Then $\G=\NPB(\{A_s\}_{s\in\Sigma_{n+1}})$, as required.

(3)$\Rightarrow$(2) From the proof of (2)$\Rightarrow$(3), we know that $\G=(sK_1\cup\G'_0)*(tK_1\cup\G'_1)$, where $\G'_0,\G'_1$ satisfy (3). Since $\G'_0,\G'_1$ satisfy (2) by induction, it follows that $\G$ satisfies (2).
\end{proof}

The following proposition characterizes the complete multipartite graphs as a subclass of non-prefix binary graphs.
\begin{proposition}
Let $\G$ be a connected graph. Then the following conditions are equivalent:
\begin{itemize}
\item[(1)]Every two disjoint edges of $\G$ lie on a square,
\item[(2)]Every disjoint vertex and edge of $\G$ are adjacent,
\item[(3)]$\G$ is a complete multipartite graph.
\end{itemize}
\end{proposition}
\begin{proof}
(3)$\Rightarrow$(1)$\Rightarrow$(2) It is obvious.

(2)$\Rightarrow$(3) Clearly, $\G$ has no induced cycles of length greater than four. We have two cases to consider:

(i) $\G$ is bipartite with bipartition $(U,V)$. Suppose $\G$ is not a complete bipartite graph and $uv\notin E(\G)$ with $u\in U$ and $v\in V$. Since $\G$ is connected $uv'\in E(\G)$ for some $v'\in V$. Then $v$ is not adjacent to the vertices of the edge $uv'$, a contradiction. Thus $\G$ is a complete bipartite graph.

(ii) $\G$ is not bipartite. Then $\G$ has a triangle with an edge $uv$. Let $\G':=\G-\{u,v\}$ and 
\begin{align*}
V_1&:=N_{\G'}(u)\setminus N_{\G'}(v),\\
V_2&:=N_{\G'}(v)\setminus N_{\G'}(u),\\
V_3&:=N_{\G'}(u)\cap N_{\G'}(v).
\end{align*}
Since every vertex of $\G'$ is adjacent to $uv$, the sets $V_1,V_2,V_3,\{u,v\}$ partition the vertex set of $\G$. Clearly, $\G[V_i]$ satisfies part (2) for $i=1,2,3$. We show that $V_i$ is an independent set in $\G$ for $i=1,2$. Suppose on the contrary that $a,b\in V_i$ are adjacent for some $i\in\{1,2\}$, say $i=1$. Then $v$ is not adjacent to $ab$ contradicting the assumption. Thus $V_1$ and $V_2$ are independent sets in $\G$.

Next, we show that every vertex of $V_i$ is adjacent to every vertex of $V_j$, for all $1\leq i<j\leq3$. Let $v_i\in V_i$ for $i=1,2,3$. Then the assumption on pair of edges $\{uv_1,vv_3\}$, $\{uv_3,vv_2\}$, and $\{uv_1,vv_2\}$ shows that $v_1,v_2,v_3$ are pairwise adjacent. 

Since $\G[V_3]$ satisfies part (2), it is complete multipartite with parts $U_3,\ldots,U_k$. Let $U_1:=V_1\cup\{v\}$ and $U_2:=V_2\cup\{u\}$. Then $U_1,\ldots,U_k$ is a partition of $V(\G)$ into independent sets where any two vertices from distinct $U_i$ and $U_j$ are adjacent. Therefore, $\G$ is a complete multipartite graph.
\end{proof}

From the definition, it is evident that $\kappa_{\G}\left(n,|E(\G)|\right)$ is the minimum size of an edge-cover of $n*\G$ by forests. Accordingly, the multigraph version of the Nash-Williams' Theorem below yields a formula to compute $\kappa_{\G}\left(n,|E(\G)|\right)$.
\begin{theorem}[Reiher and Sauermann, \cite{cr-ls}]
If $\G$ is a multigraph, then the minimum size of an edge-cover of $\G$ by forests is
\[\rho(\G)=\max_{\substack{X\subseteq V(\G)\\|X|\geq2}}\left\lceil\frac{|E(\G[X])|}{|X|-1}\right\rceil.\]
\end{theorem}

In what follows, we analyze the formula in Nash-Williams' Theorem to see when the maximum is attained.
\begin{definition}
Let $\G$ be a graph and $0\leq\varepsilon\leq1$. A subset $X$ of $V(\G)$ is an \textit{$\varepsilon$-set} of $\G$ if there exists a vertex $u\in V(\G)\setminus X$ satisfying $\deg_X(u):=|N_\G(u)\cap X|\geq\varepsilon|X|$. The graph $\G$ is an \textit{$\varepsilon$-graph} if any proper set $X$ of vertices of $\G$ is an $\varepsilon$-set.
\end{definition}
\begin{lemma}\label{rho(G) for 1/2-sets}
Let $\G$ be a connected graph and $\mathcal{X}$ be the family of all $\frac{1}{2}$-sets of $\G$. Then 
\[\rho(\G)=\max_{\substack{X\subseteq V(\G)\\X\notin\mathcal{X}}}\left\lceil\frac{|E(\G[X])|}{|X|-1}\right\rceil.\]
In particular, 
\[\rho(\G)=\left\lceil\frac{|E(\G)|}{|V(\G)|-1}\right\rceil\]
if $\G$ is a $\frac{1}{2}$-graph.
\end{lemma}
\begin{proof}
Recall that $\rho(\G)$ is the maximum of $\lceil |E(\G[X])/(|X|-1)\rceil$ taken over all subsets $X$ of $V(\G)$ with at least two elements.

Let $X$ be a subset of $V(\G)$. If $X\in\mathcal{X}$, then $\deg_X(u)\geq|X|/2$ for some $u\in V(\G)\setminus X$. Thus 
\[\deg_X(u)\geq\frac{|X|}{2}\geq\frac{|E(\G[X])|}{|X|-1}\]
from which it follows that
\[\frac{|E(\G[X\cup\{u\}])|}{|X\cup\{u\}|-1}=\frac{|E(\G[X])|+\deg_X(u)}{|X|}\geq\frac{|E(\G[X])|}{|X|-1}.\]
Hence, $X$ can be dropped from the family of sets under maximum.

Since every proper subset of $V(\G)$ is a $\frac{1}{2}$-set of the $\frac{1}{2}$-graph $\G$, the second statement follows immediately.
\end{proof}
\begin{remark}
If $\G=\G_1*\G_2$, then $\rho(\G)$ cannot be expressed as a function of $\rho(\G_1)$ and $\rho(\G_2)$. Indeed, if $\G_1=K_m$ and $\G_2=tK_1\cup K_n$ with $m\geq n$, then $\rho(\G_1)=\ceil{m/2}$ and $\rho(\G_2)=\ceil{n/2}$ while
\[\rho(\G_1*\G_2)=\wceil{\frac{\binom{m+n}{2}+mt}{m+n+t-1}}\]
is a function of $t$ taking any value from $\ceil{(m+n)/2}$ to $m$.
\end{remark}
\begin{lemma}\label{non-1/2-subsets of G1*G2}
Let $\G=\G_1*\G_2$ with $|\G_1|\leq|\G_2|$. If $X\subseteq V(\G)$ is not a $\frac{1}{2}$-set of $\G$, then $V(\G_1)\subseteq X$.
\end{lemma}
\begin{proof}
Let $X_i:=X\cap V(\G_i)$ for $i=1,2$. Suppose $|X_i|\leq|X_j|$ for $\{i,j\}=\{1,2\}$. If $X_i\neq V(\G_i)$ and $u\in V(\G_i)\setminus X_i$, then $\deg_X(u)\geq|X|/2$ contradicting the fact that $X$ is not a $\frac{1}{2}$-set of $\G$. Thus $X_i=V(\G_i)$ from which the result follows.
\end{proof}
\begin{proposition}\label{kappa of NPB-graphs}
Let $\G=\NPB(\{A_s\}_{s\in\Sigma_n})$ be a non-prefix binary graph. Let $B_s:=\cup_{ss'\in\Sigma_n}A_{ss'}$ for all $s\in\Sigma_n$. Suppose $|B_{1^i0}|\leq|B_{1^i1}|$ for all $i=0,\ldots,n-1$. Then
\[\rho(n*\G)=\max_{X'\subseteq X\subseteq V(\G)}\left\lceil\frac{n|E(\G[X])|}{|X|-1}\right\rceil\]
where $X':=B_0\cup B_{10}\cup\cdots\cup B_{1^{n-1}0}$. In addition, 
\[\rho(n*\G)=\left\lceil\frac{n|E(\G)|}{|V(\G)|-1}\right\rceil\]
if $\G$ is a complete multipartite graph.
\end{proposition}
\begin{proof}
Without loss of generality, we may assume that $n=1$. Let $\G_0:=\G$ and $\G_i:=\G[B_{1^i0}\setminus A_{1^i0}]$ for $i=1,\ldots,n-1$. For $X\subseteq V(\G)$ put $X_i:=X\cap V(\G_i)$ for $i=0,\ldots,n-1$.

A simple verification shows that 
\[\deg_{X_{i-1}}(u)\geq\frac{|E(\G_{i-1}[X_{i-1}])|}{|X_{i-1}|-1}\]
if $\deg_{X_i}(u)\geq|E(\G_i[X_i])|/(|X_i|-1)$ for some $u\in V(\G_i)\setminus X_i$.

Suppose $|E(\G[X])|/(|X|-1)$ takes the maximum value and that $X$ is maximal with this property. We show that $X_i$, as define above, is not a $\frac{1}{2}$-set of $\G_i$. Suppose on the contrary that $X_i$ is a $\frac{1}{2}$-set of $\G_i$ for some $0\leq i<n$. Then 
\[\deg_{X_i}(u)\geq\frac{|X_i|}{2}\geq\frac{|E(\G_i[X_i])|}{|X_i|-1}\]
for some $u\in V(\G_i)\setminus X_i$.  The above argument shows that 
\[\deg_X(u)=\deg_{X_0}(u)\geq\frac{|E(\G_0[X_0])|}{|X_0|-1}=\frac{|E(\G[X])|}{|X|-1}.\]
Thus 
\[\frac{|E(\G[X\cup\{u\}])|}{|X\cup\{u\}|-1}=\frac{|E(\G[X])|+\deg_X(u)}{|X|}\geq\frac{|E(\G[X])|}{|X|-1}\]
contradicting the maximality of $X$. Therefore, $X_i$ is not a $\frac{1}{2}$-set of $\G_i$ ($i=0,\ldots,n-1$) so that $B_{1^i0}\subseteq X_i$ by Lemma \ref{non-1/2-subsets of G1*G2}. Hence, $X'\subseteq X$, as required.

Now, we show the second part of the proposition. Let $\G$ be a complete multipartite graph with maximal independent sets $V_1,\ldots,V_m$. Let $X$ be a proper subset of $V(\G)$ that is not a $\frac{1}{2}$-set of $\G$. A simple verification shows, after relabeling of $V_i$'s, that $X=V_1\cup\cdots\cup V_{m-1}\cup V'_m$ for some proper subset $V'_m$ of $V_m$ such that $|V_1|+\cdots+|V_{m-1}|<|V'_m|$. Let $u\in V_m\setminus V'_m$. Then
\[\deg_X(u)=\sum_{i=1}^{m-1}|V_i|\geq\frac{\sum_{1\leq i<j<m}^{m-1}|V_i||V_j|+\sum_{i=1}^{m-1}|V_i||V'_m|}{\sum_{i=1}^{m-1}|V_i|+|V'_m|}=\frac{|E(\G[X])|}{|X|-1},\]
from which it follows that
\[\frac{|E(\G[X\cup\{u\}])}{|X\cup\{u\}|-1}\geq\frac{|E(\G[X])|}{|X|-1}.\]
The above argument, in conjunction with Lemma \ref{rho(G) for 1/2-sets}, shows that $\frac{|E(\G[X])|}{|X|-1}$ takes the maximum value when $X=V(\G)$, as required.
\end{proof}
\begin{remark}
With the notation as in Proposition \ref{kappa of NPB-graphs}:
\begin{itemize}
\item[(1)]The assumption $|B_{1^i0}|\leq|B_{1^i1}|$ for all $i=0,\ldots,n-1$ is not restrictive as it is attained after a suitable relabeling of $A_s$'s.
\item[(2)]The set $X$ for which $\ceil{n|E(\G[X])|/(|X|-1)}$ takes the maximum value in Proposition \ref{kappa of NPB-graphs} need not be equal to either $X'$ or $V(\G)$. Indeed, $\G=K_3*(K_1\cup (K_3*(5K_1\cup K_3)))$ is a non-prefix binary graph with $|X'|=9$ and $|V(\G)|=15$. Moreover, 
\[\wceil{\frac{|E(\G[X'])|}{|X'|-1}}=\wceil{\frac{|E(\G)|}{|V(\G)|-1}}=5\]
while $\kappa_{\G}\left(1,|E(\G)|\right)=\ceil{|E(\G[X])|/(|X|-1)}=6$ for a subset $X$ of $V(\G)$ with $|X|=14$.
\item[(3)]We know from Proposition \ref{kappa of NPB-graphs} that $\rho(\G)$ takes its maximum value on the whole vertex set when $\G$ is a complete multipartite graph. This is also true for $\frac{1}{2}$-graphs (see Lemma \ref{rho(G) for 1/2-sets}) while complete multipartite graphs are not necessarily $\frac{1}{2}$-graphs.
\end{itemize}
\end{remark}

We conclude this section with introducing an alternate way to compute $\kappa_{\G}(1,|E(\G)|)$ when $\G$ is a complete bipartite graph.
\begin{definition}
Let $\calP_X$ denote the path $x_1\sim\cdots\sim x_p$ for any set $X=\{x_1,\ldots,x_p\}$ of positive integers with $x_1<\cdots<x_p$. Let $k,m,n$ be positive integers, and $\{\calP_i\}_{i=1}^m$ be a family of ordered $k$-partitions of $[n]$, where $\calP_i=\{P_{i,j}\}_{j=1}^k$ for all $i=1,\ldots,n$. For $1\leq j\leq k$, let $\G_j(\{\calP_i\}_{i=1}^m)=([n],E_j)$ be the multigraph on $[n]$ with $E_j:=\cup_{i=1}^mE(\calP_{P_{i,j}})$ as a multiset. We say that $\{\calP_i\}_{i=1}^n$ is an \textit{acyclic family of ordered $k$-partitions} of $[n]$ if $\G_j(\{\calP_i\}_{i=1}^n)=([n],E_j)$ is acyclic as multigraph for all $j=1,\ldots,k$.
\end{definition}
\begin{proposition}\label{kappa_K(a,b)(1,ab)<=min(k:[b] has an acyclic family of ordered k-partitions of size a)}
For $a\leq b$, 
\[\kappa_{K_{a,b}}(1,ab)=\min\{k\colon [b]\text{ has an acyclic family of ordered $k$-partitions of size $a$}\}.\]
\end{proposition}
\begin{proof}
Let $V(K_{a,b})=U\cup V$ with $U=\{u_1,\ldots,u_a\}$ and $V=\{v_1,\ldots,v_b\}$ be the bipartition of $K_{a,b}$. For a family $\calE_1,\ldots,\calE_k$ of edges, let $P_{i,j}$ be the set of all $v_l$ with $u_iv_l\in\calE_j$ for all $1\leq i\leq a$ and $1\leq j\leq k$. Let $\calE'_i=\{\{e\}\colon e\in\calE_i\}$ for all $i=1,\ldots,k$. Let $\calP_i=\{P_{i,j}\}_{j=1}^k$ for $i=1,\ldots,a$. Clearly, $\calE'_1,\ldots,\calE'_k$ is a family of AFE-sets if and only if $\G_j(\{\calP_i\}_{i=1}^a)=([b],E_j)$ are acyclic for $j=1,\ldots,k$. Hence the result follows immediately from \eqref{kappa definition}.
\end{proof}
\section{Special graphs}
The aim of this section is to compute the pmd of grid graphs that is the Cartesian products of paths and cycles. One observe that
\begin{align*}
\pmd(P_2\square P_n)&=3,\tag{Theorem \ref{pmd(G1 square G2): forest decomposition 1}}\\
\pmd(P_2\square C_n)&=4\quad (n\neq4),\tag{\cite[Proposition 4.1]{mfdg-sg-aayp}}\\
\pmd(P_2\square C_4)&=5.\tag{\cite[Proposition 4.1]{mfdg-sg-aayp}}
\end{align*}

Let $m,n\geq3$. We know from \cref{pmd(T1 square ... square Tn)} that $\pmd(P_m\square P_n)=4$. Also, one can verify that 
\begin{align*}
4\leq\pmd(P_m\square C_n)&\leq5+\varepsilon_n,\tag{\cref{pmd(G1 square G2): forest decomposition 1}}\\
4\leq\pmd(C_m\square C_n)&\leq8,\tag{\cref{pmd(G_1 O G_2): AFE-cover}}
\end{align*}
where $\varepsilon_n=n-2[n/2]$ is the parity of $n$. In the sequel, we improve the above inequalities.
\begin{proposition}
For all $m,n\geq3$, 
\[4\leq\pmd(P_m\square C_n)\leq5.\]
In addition, $\pmd(P_m\square C_n)=4$ if either $n$ is even and $n\geq4(m-1)$ or $n$ is odd and $n>2m$.
\end{proposition}
\begin{proof}
Clearly, $\pmd(P_m\square C_n)\geq\Delta(P_m\square C_n)=4$. Assume $V(P_m\square C_n)=[m]\times[n]$, where $V(P_m)=[m]$ and $V(C_n)=[n]$. Let 
\[M_1:=\{\{i,i+1\}\times j\colon\quad i<m,\ j\leq n-\varepsilon_n,\ \varepsilon_i=\varepsilon_j\}\]
and 
\[M_2=M_1+(0,1):=\{\{(u,v+1),(u',v'+1)\}\colon\quad \{(u,v),(u',v')\}\in M_1\}\]
(see \cref{Pm O Cn: n odd} and \cref{Pm O Cn: n even}). It is easy to see that $M_1$ and $M_2$ are positive matchings of $P_m\square C_n$, and that $P_m\square C_n-M_1-M_2$ is either a union of $n$-cycles (if $n$ is even) or it is a path-like cactus graph (if $n$ is odd). Note that a cactus graph is path-like if the graph obtained after contraction of its cycles is a path. Clearly, $\pmd(P_m\square C_n-M_1-M_2)=3$ from which it follows that $\pmd(P_m\square C_n)\leq5$.

To complete the proof, let $P(i,j,d)$ denote the path
\[(i,1),(i,2)\ldots,(i,d),(i+1,d),\ldots,(j-1,d),(j,d),\ldots,(j,2),(j,1)\]
for all $1\leq i<j\leq n$ and $1\leq d\leq m$. First assume that $n>2m$ is odd. Let $M_1$ and $M_2$ be matchings of $P_m\square C_n$ consisting of alternating edges of the paths $P(i,n+1-i,m+1-i)$ ($i=1,\ldots,m$) as in \cref{Pm O Cn: n odd; n>2m}. It is evident that $M_1$ and $M_2$ are both positive matchings of $P_m\square C_n$ and that $P_m\square C_n-M_1-M_2$ is a union of paths. Thus $\pmd(P_m\square C_n)\leq4$, which indeed implies that $\pmd(P_m\square C_n)=4$. Finally, assume that $n\geq4(m-1)$ is even. Let $M_1$ and $M_2$ be matchings of $P_m\square C_n$ consisting of alternating edges of the paths $P(i,n/2+1-i,m-i)$ and $P(n/2+i,n+1-i,m-i)$ ($i=1,\ldots,m-1$) as well as the paths $(2,m),\ldots,(n/2+1,m)$ and $(n/2+2,m),\ldots,(n,m),(1,m)$ as in \cref{Pm O Cn: n even; n>=4(m-1)}. Again $M_1$ and $M_2$ are both positive matchings of $P_m\square C_n$ and that $P_m\square C_n-M_1-M_2$ is a union of paths. Thus $\pmd(P_m\square C_n)\leq4$, which indeed implies that $\pmd(P_m\square C_n)=4$.
\end{proof}

\begin{figure}[h]
\begin{subfigure}{0.45\textwidth}
\centering
\begin{tikzpicture}[scale=0.5]
\tikzmath{
function vertex(\c, \p){
	return int(\c + (\p - 1) * \n);
};
\m = 5;
\n = 11;
}

\foreach \c in {1,...,\n}
{
	\draw [gray!50] ({cos(\c*360/\n)}, {sin(\c*360/\n)})--({\m*cos(\c*360/\n)}, {\m*sin(\c*360/\n)});

	\foreach \p in {1,...,\m}
	{
		\draw [gray!50] ({\p*cos(\c*360/\n)}, {\p*sin(\c*360/\n)})--({\p*cos((\c+1)*360/\n)}, {\p*sin((\c+1)*360/\n)});
	}
}

\foreach \c in {1,...,\n}
{
	\foreach \p in {1,...,\m}
	{
		\tikzmath{\k = vertex(\c,\m + 1 - \p);}
		\node [draw, fill=white, circle, inner sep=1pt] (\k) at ({\p*cos((\c-1)*360/\n)}, {\p*sin((\c-1)*360/\n)}) {};
	}
	
	\node () at ({(\m + 0.6)*cos((\c-1)*360/\n)}, {(\m+0.6)*sin((\c-1)*360/\n)}) {\tiny{\c}};
}

\def\matching{{
{{1,1},{1,2}},
{{1,3},{1,4}},
{{2,2},{2,3}},
{{2,4},{2,5}},
{{3,1},{3,2}},
{{3,3},{3,4}},
{{4,2},{4,3}},
{{4,4},{4,5}},
{{5,1},{5,2}},
{{5,3},{5,4}},
{{6,2},{6,3}},
{{6,4},{6,5}},
{{7,1},{7,2}},
{{7,3},{7,4}},
{{8,2},{8,3}},
{{8,4},{8,5}},
{{9,1},{9,2}},
{{9,3},{9,4}},
{{10,2},{10,3}},
{{10,4},{10,5}}
}}

\foreach \i in {0,...,19}
{
	\tikzmath{\u = vertex(\matching[\i][0][0],\matching[\i][0][1]);
	\v = vertex(\matching[\i][1][0],\matching[\i][1][1]);}
	\draw [line width=0.75mm] (\u)--(\v);
}

\def\matching{{
{{2,1},{2,2}},
{{2,3},{2,4}},
{{3,2},{3,3}},
{{3,4},{3,5}},
{{4,1},{4,2}},
{{4,3},{4,4}},
{{5,2},{5,3}},
{{5,4},{5,5}},
{{6,1},{6,2}},
{{6,3},{6,4}},
{{7,2},{7,3}},
{{7,4},{7,5}},
{{8,1},{8,2}},
{{8,3},{8,4}},
{{9,2},{9,3}},
{{9,4},{9,5}},
{{10,1},{10,2}},
{{10,3},{10,4}},
{{11,2},{11,3}},
{{11,4},{11,5}}
}}

\foreach \i in {0,...,19}
{
	\tikzmath{\u = vertex(\matching[\i][0][0],\matching[\i][0][1]);
	\v = vertex(\matching[\i][1][0],\matching[\i][1][1]);}
	\draw [line width=0.75mm, gray!75] (\u)--(\v);
}
\end{tikzpicture}
\caption{$n$ is odd}\label{Pm O Cn: n odd}
\end{subfigure}
\hfill
\begin{subfigure}{0.45\textwidth}
\centering
\begin{tikzpicture}[scale=0.5]
\tikzmath{
function vertex(\c, \p){
	return int(\c + (\p - 1) * \n);
};
\m = 5;
\n = 12;
}

\foreach \c in {1,...,\n}
{
	\draw [gray!50] ({cos(\c*360/\n)}, {sin(\c*360/\n)})--({\m*cos(\c*360/\n)}, {\m*sin(\c*360/\n)});

	\foreach \p in {1,...,\m}
	{
		\draw [gray!50] ({\p*cos(\c*360/\n)}, {\p*sin(\c*360/\n)})--({\p*cos((\c+1)*360/\n)}, {\p*sin((\c+1)*360/\n)});
	}
}

\foreach \c in {1,...,\n}
{
	\foreach \p in {1,...,\m}
	{
		\tikzmath{\k = vertex(\c,\m + 1 - \p);}
		\node [draw, fill=white, circle, inner sep=1pt] (\k) at ({\p*cos((\c-1)*360/\n)}, {\p*sin((\c-1)*360/\n)}) {};
	}
	
	\node () at ({(\m+0.6)*cos((\c-1)*360/\n)}, {(\m+0.6)*sin((\c-1)*360/\n)}) {\tiny{\c}};
}

\def\matching{{
{{1,1},{1,2}},
{{1,3},{1,4}},
{{2,2},{2,3}},
{{2,4},{2,5}},
{{3,1},{3,2}},
{{3,3},{3,4}},
{{4,2},{4,3}},
{{4,4},{4,5}},
{{5,1},{5,2}},
{{5,3},{5,4}},
{{6,2},{6,3}},
{{6,4},{6,5}},
{{7,1},{7,2}},
{{7,3},{7,4}},
{{8,2},{8,3}},
{{8,4},{8,5}},
{{9,1},{9,2}},
{{9,3},{9,4}},
{{10,2},{10,3}},
{{10,4},{10,5}},
{{11,1},{11,2}},
{{11,3},{11,4}},
{{12,2},{12,3}},
{{12,4},{12,5}},
}}

\foreach \i in {0,...,23}
{
	\tikzmath{\u = vertex(\matching[\i][0][0],\matching[\i][0][1]);
	\v = vertex(\matching[\i][1][0],\matching[\i][1][1]);}
	\draw [line width=0.75mm] (\u)--(\v);
}

\def\matching{{
{{1,2},{1,3}},
{{1,4},{1,5}},
{{2,1},{2,2}},
{{2,3},{2,4}},
{{3,2},{3,3}},
{{3,4},{3,5}},
{{4,1},{4,2}},
{{4,3},{4,4}},
{{5,2},{5,3}},
{{5,4},{5,5}},
{{6,1},{6,2}},
{{6,3},{6,4}},
{{7,2},{7,3}},
{{7,4},{7,5}},
{{8,1},{8,2}},
{{8,3},{8,4}},
{{9,2},{9,3}},
{{9,4},{9,5}},
{{10,1},{10,2}},
{{10,3},{10,4}},
{{11,2},{11,3}},
{{11,4},{11,5}},
{{12,1},{12,2}},
{{12,3},{12,4}},
}}

\foreach \i in {0,...,23}
{
	\tikzmath{\u = vertex(\matching[\i][0][0],\matching[\i][0][1]);
	\v = vertex(\matching[\i][1][0],\matching[\i][1][1]);}
	\draw [line width=0.75mm, gray!75] (\u)--(\v);
}

\end{tikzpicture}
\caption{$n$ is even}\label{Pm O Cn: n even}
\end{subfigure}
\caption{Positive matchings $M_1$ (black) and $M_2$ (gray) of $P_m\square C_n$}
\end{figure}
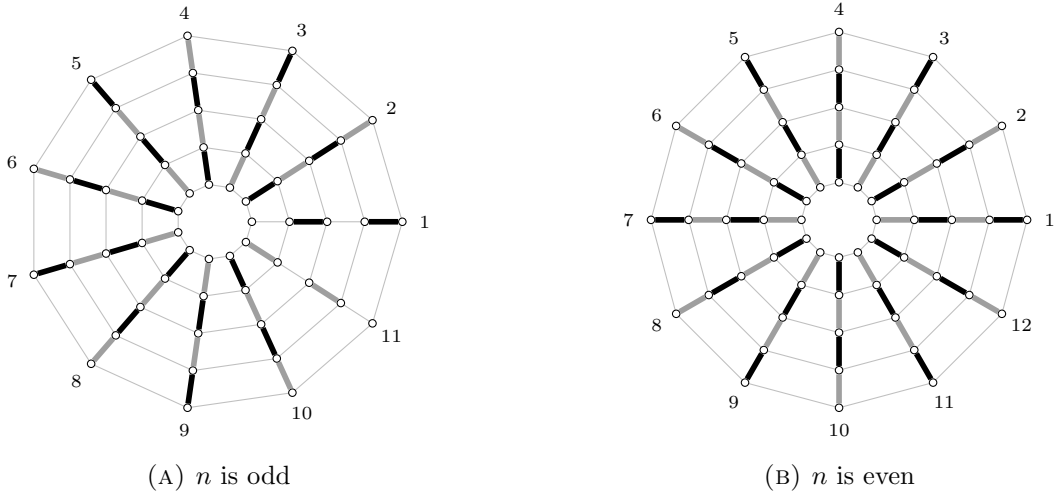
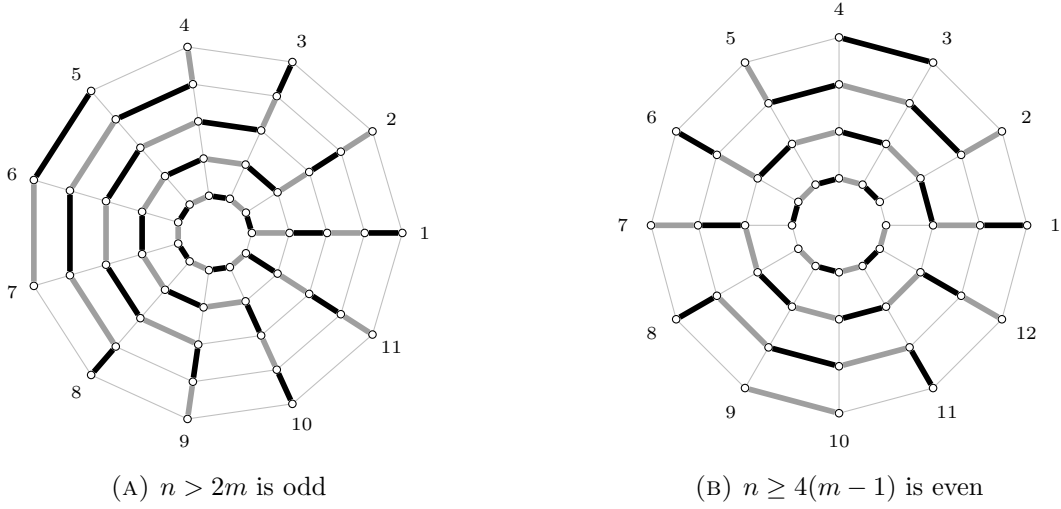
\begin{figure}[h]
\begin{subfigure}{0.45\textwidth}
\centering
\begin{tikzpicture}[scale=0.5]
\tikzmath{
function vertex(\c, \p){
	return int(\c + (\p - 1) * \n);
};
\m = 5;
\n = 11;
}

\foreach \c in {1,...,\n}
{
	\draw [gray!50] ({cos(\c*360/\n)}, {sin(\c*360/\n)})--({\m*cos(\c*360/\n)}, {\m*sin(\c*360/\n)});

	\foreach \p in {1,...,\m}
	{
		\draw [gray!50] ({\p*cos(\c*360/\n)}, {\p*sin(\c*360/\n)})--({\p*cos((\c+1)*360/\n)}, {\p*sin((\c+1)*360/\n)});
	}
}

\foreach \c in {1,...,\n}
{
	\foreach \p in {1,...,\m}
	{
		\tikzmath{\k = vertex(\c,\p);}
		\node [draw, fill=white, circle, inner sep=1pt] (\k) at ({\p*cos((\c-1)*360/\n)}, {\p*sin((\c-1)*360/\n)}) {};
	}
	
	\node () at ({(\m + 0.6)*cos((\c-1)*360/\n)}, {(\m+0.6)*sin((\c-1)*360/\n)}) {\tiny{\c}};
}

\def\matching{{
{{1,5},{1,4}},
{{1,3},{1,2}},
{{1,1},{2,1}},
{{3,1},{4,1}},
{{5,1},{6,1}},
{{7,1},{8,1}},
{{9,1},{10,1}},
{{11,1},{11,2}},
{{11,3},{11,4}},
{{2,4},{2,3}},
{{2,2},{3,2}},
{{4,2},{5,2}},
{{6,2},{7,2}},
{{8,2},{9,2}},
{{10,2},{10,3}},
{{10,4},{10,5}},
{{3,5},{3,4}},
{{3,3},{4,3}},
{{5,3},{6,3}},
{{7,3},{8,3}},
{{9,3},{9,4}},
{{4,4},{5,4}},
{{6,4},{7,4}},
{{8,4},{8,5}},
{{5,5},{6,5}}
}}

\foreach \i in {0,...,24}
{
	\tikzmath{\u = vertex(\matching[\i][0][0],\matching[\i][0][1]);
	\v = vertex(\matching[\i][1][0],\matching[\i][1][1]);}
	\draw [line width=0.75mm] (\u)--(\v);
}

\def\matching{{
{{1,4},{1,3}},
{{1,2},{1,1}},
{{2,1},{3,1}},
{{4,1},{5,1}},
{{6,1},{7,1}},
{{8,1},{9,1}},
{{10,1},{11,1}},
{{11,2},{11,3}},
{{11,4},{11,5}},
{{2,5},{2,4}},
{{2,3},{2,2}},
{{3,2},{4,2}},
{{5,2},{6,2}},
{{7,2},{8,2}},
{{9,2},{10,2}},
{{10,3},{10,4}},
{{3,4},{3,3}},
{{4,3},{5,3}},
{{6,3},{7,3}},
{{8,3},{9,3}},
{{9,4},{9,5}},
{{4,5},{4,4}},
{{5,4},{6,4}},
{{7,4},{8,4}},
{{6,5},{7,5}}
}}

\foreach \i in {0,...,24}
{
	\tikzmath{\u = vertex(\matching[\i][0][0],\matching[\i][0][1]);
	\v = vertex(\matching[\i][1][0],\matching[\i][1][1]);}
	\draw [line width=0.75mm, gray!75] (\u)--(\v);
}
\end{tikzpicture}
\caption{$n>2m$ is odd}\label{Pm O Cn: n odd; n>2m}
\end{subfigure}
\hfill
\begin{subfigure}{0.45\textwidth}
\centering
\begin{tikzpicture}[scale=0.625]
\tikzmath{
function vertex(\c, \p){
	return int(\c + (\p - 1) * \n);
};
\m = 4;
\n = 12;
}

\foreach \c in {1,...,\n}
{
	\draw [gray!50] ({cos(\c*360/\n)}, {sin(\c*360/\n)})--({\m*cos(\c*360/\n)}, {\m*sin(\c*360/\n)});

	\foreach \p in {1,...,\m}
	{
		\draw [gray!50] ({\p*cos(\c*360/\n)}, {\p*sin(\c*360/\n)})--({\p*cos((\c+1)*360/\n)}, {\p*sin((\c+1)*360/\n)});
	}
}

\foreach \c in {1,...,\n}
{
	\foreach \p in {1,...,\m}
	{
		\tikzmath{\k = vertex(\c,\p);}
		\node [draw, fill=white, circle, inner sep=1pt] (\k) at ({\p*cos((\c-1)*360/\n)}, {\p*sin((\c-1)*360/\n)}) {};
	}
	
	\node () at ({(\m+0.6)*cos((\c-1)*360/\n)}, {(\m+0.6)*sin((\c-1)*360/\n)}) {\tiny{\c}};
}

\def\matching{{
{{1,4},{1,3}},
{{1,2},{2,2}},
{{3,2},{4,2}},
{{5,2},{6,2}},
{{6,3},{6,4}},
{{7,3},{7,2}},
{{8,2},{9,2}},
{{10,2},{11,2}},
{{12,2},{12,3}},
{{2,3},{3,3}},
{{4,3},{5,3}},
{{8,4},{8,3}},
{{9,3},{10,3}},
{{11,3},{11,4}},
{{3,4},{4,4}},
{{2,1},{3,1}},
{{4,1},{5,1}},
{{6,1},{7,1}},
{{9,1},{10,1}},
{{11,1},{12,1}}
}}

\foreach \i in {0,...,19}
{
	\tikzmath{\u = vertex(\matching[\i][0][0],\matching[\i][0][1]);
	\v = vertex(\matching[\i][1][0],\matching[\i][1][1]);}
	\draw [line width=0.75mm] (\u)--(\v);
}

\def\matching{{
{{1,3},{1,2}},
{{2,2},{3,2}},
{{4,2},{5,2}},
{{6,2},{6,3}},
{{7,4},{7,3}},
{{7,2},{8,2}},
{{9,2},{10,2}},
{{11,2},{12,2}},
{{12,3},{12,4}},
{{2,4},{2,3}},
{{3,3},{4,3}},
{{5,3},{5,4}},
{{8,3},{9,3}},
{{10,3},{11,3}},
{{9,4},{10,4}},
{{3,1},{4,1}},
{{5,1},{6,1}},
{{8,1},{9,1}},
{{10,1},{11,1}},
{{12,1},{1,1}}
}}

\foreach \i in {0,...,19}
{
	\tikzmath{\u = vertex(\matching[\i][0][0],\matching[\i][0][1]);
	\v = vertex(\matching[\i][1][0],\matching[\i][1][1]);}
	\draw [line width=0.75mm, gray!75] (\u)--(\v);
}

\end{tikzpicture}
\caption{$n\geq 4(m-1)$ is even}\label{Pm O Cn: n even; n>=4(m-1)}
\end{subfigure}
\caption{Positive matchings $M_1$ (black) and $M_2$ (gray) of $P_m\square C_n$}
\end{figure}

A direct calculation reveals that 
\[\pmd(P_3\square C_3)=\pmd(P_3\square C_4)=\pmd(P_3\square C_5)=\pmd(P_3\square C_6)=5.\]
This motivates us to pose the following:
\begin{conjecture}
If $m,n\geq3$, then $\pmd(P_m\square C_n)=5$ whenever $n$ is even and $n<4(m-1)$ or $n$ is odd and $n<2m$.
\end{conjecture}

In order to compute the pmd of the Cartesian product of two cycles, we need to compute the pmd of a class of graphs we call them here as circular wall graphs. The following lemma gives us a lower bound for the pmd of regular graphs including circular wall graphs as well as the Cartesian product of cycles.
\begin{lemma}\label{Regular graphs}
If $\G$ is an $r$-regular graph ($r\geq2$), then $\pmd(\G)\geq r+1$.
\end{lemma}
\begin{proof}
We know that $\pmd(\G)\geq\Delta(\G)=r$. Suppose on the contrary that $\pmd(\G)=r$ and $M_1,\ldots,M_r$ is a pmd of $\G$. Then $M_1,\ldots,M_r$ are vertex covers of $\G$ so that $\G':=\G-M_1-\cdots-M_{r-2}$ is a $2$-regular graph. Since $\G'$ is a union of cycles, it follows that $\pmd(\G')=3$ contradicting the fact that $M_{r-1},M_r$ is a pmd of $\G'$. Therefore, $\pmd(\G)\geq r+1$.
\end{proof}
\begin{definition}
Let $m\geq2$ be a positive integer and $n\geq4$ be even. The \textit{circular wall graph} $CW(m,n)$ is defined as the subgraph of $P_m\square C_n$ with vertex set $[m]\times[n]$ by removing edges $\{(i,j),(i+1,j)\}$ whenever $i+j$ is odd.
\end{definition}
\begin{proposition}\label{pmd of circular wall graphs}
If $m\geq2$, and $n\geq4$ is even, then
\[\pmd(CW(m,n))=\begin{cases}
3,&n>2m,\\
4,&n\leq 2m.
\end{cases}\]
\end{proposition}
\begin{proof}
Let $\G:=CW(m,n)$. First assume that $n\geq2m+2$. Let 
\[M:=\{\{(1,1),(1,n)\}\}\cup\{\{(1,2i-1),(2,2i-1)\}\colon\ 4\leq 2i\leq n\}\]
and put
\[M_1=M\cup(M+(1,1))\cup\cdots\cup(M+(m-1, m-1))\cup\{\{(m-1,m),(m,m)\}\},\]
where all additions are taken modulo $n$ (see \cref{CW(m;n): M1; n>=2m+2}). We note that $M+(a,b)$ is defined as the set of all edges $\{(i+a,j+b),(i'+a,j'+b)\}$ with $\{(i,j),(i',j')\}\in M$. Since $n\geq2m+2$, a simple verification shows that $M_1$ is a positive matching in $\G$ such that $\G-M_1$ is a Hamiltonian path. Thus $\pmd(\G)=3$.

Now, assume that $n\leq2m$. If $m=2$ then $n=4$ and $\pmd(\G)=4$. Hence, assume that $m\geq3$. If $M_1$ is the set of all spokes of $\G$, then $M_1$ is a positive matching in $\G$ and $\G-M_1=m C_n$. Thus $\pmd(\G)\leq4$. We show that $\pmd(\G)\neq3$. Suppose on the contrary that $\pmd(\G)=3$ and $M_1,M_2,M_3$ is a pmd of $\G$.

(1) If $i<m$, then $\{(i,j),(i,j+1)\}\in M_1$ if and only if either $i$ is odd and
\[\{(i+1,j-1),(i+1,j)\}\in M_1,\]
or $i$ is even and
\[\{(i+1,j+1),(i+1,j+2)\}\in M_1.\]
To prove the claim, assume without loss of generality that $i=1$ and $j=1$. Suppose on the contrary that $\{(1,1),(1,2)\}\in M_1$ but $\{(2,n),(2,1)\}\notin M_1$. Then $\{(2,1),(2,2)\}\in M_1$ since $M_1$ covers all cubic vertices. Suppose $\{(1,t),(1,t+1)\},\{(2,t),(2,t+1)\}\in M_1$ for some $t$. We show that $\{(1,t+2),(1,t+3)\},\{(2,t+2),(2,t+3)\}\in M_1$. Clearly, $\{(1,t+1),(1,t+2)\},\{(2,t+1),(2,t+2)\}\notin M_1$. Since $\G[M_1]$ does not have any alternating $6$-cycle, it follows that $\{(1,t+2),(2,t+2)\}\notin M_1$. Thus $\{(1,t+2),(1,t+3)\},\{(2,t+2),(2,t+3)\}\in M_1$, as required. Hence, we conclude that 
\[\{(1,1),(1,2)\},\{(1,3),(1,4)\},\ldots,\{(1,n-1),(1,n)\}\in M_1\]
so that $\G[M_1]$ has an alternating $n$-cycle. This contradiction shows that $\{(2,n),(2,1)\}\in M_1$ whenever $\{(1,1),(1,2)\}\in M_1$. A same argument shows that $\{(1,1),(1,2)\}\in M_1$ whenever $\{(2,n),(2,1)\}\in M_1$.

(2) $M_1$ has the following set of edges (up to symmetry)
\[C := \{\{(1,n),(1,1)\},\{(2,1),(2,2)\},\ldots,\{(m,m-1),(m,m)\}\}.\]
Indeed, the fact that $\pmd(C_n)=3$ yields $M_1$ has an edge with vertices in $\{1\}\times[n]$, say $\{(n,1),(1,1)\}$. Thus, $M_1\supseteq C$ by (1) (See \cref{CW(m;n): C and S; n<=2m}).

(3) $M_1$ does not contain any of the edges $\{(i,j),(i,j+1)\}$ with $i+j$ even.
First observe that $M_1\cap(C+(0,1))=\varnothing$ as $M_1$ is a matching. Since $\{(i,i),(i,i+1)\}\in C+(0,1)$ ($1\leq i\leq m$), we get from (1) that $\{(i-t,i+t),(i-t,i+t+1)\}\notin M_1$ for $t=1,\ldots,i-1$ by induction. Hence $\{(1,2i-1),(1,2i)\}\notin M_1$ for all $1\leq i\leq n/2$. Notice that $n/2\leq m$. Applying (1) once more, one can easily see, by induction on $i$, that $\{(i,j),(i,j+1)\}\notin M_1$ with $i+j$ even.

(4) $M_1$ has following set of edges (up to symmetry)
\[S := \{\{(1,n-1),(2,n-1)\},\{(2,n),(3,n)\},\ldots,\{(m-1,m-3),(m,m-3)\}\}.\]
We know from (3) that $M_1$ does not contain the alternating edges $\{(1,2i-1),(1,2i)\}$ ($i=1,\ldots,n/2$). Since $M_1$ covers all cubic vertices, all vertices $(1,1),(1,3),\ldots,(1,n-1)$ are covered by $M_1$. If $M_1$ contains all the edges $\{(1,2i-2),(1,2i-1)\}$ ($i=1,\ldots,n/2$), then $\G[M_1]$ has an alternating $n$-cycle, which is impossible. Thus there exists $2\leq i\leq n/2$ such that $\{(1,2i-2),(1,2i-1)\}\notin M_1$. Since $(1,2i-1)$ is cubic, we should have $\{(1,2i-1),(2,2i-1)\}\in M_1$. Since either $\{(1,2j-2),(1,2j-1)\}$ or $\{(1,2j-1),(2,2j-1)\}$ belongs to $M_1$ for any $j=1,\ldots,n/2$, we can assume without loss of generality that $\{(1,n-1),(2,n-1)\}\in M_1$ (by applying a suitable rotation). Now, by invoking (1) and induction on $t$, one can show that $\{(1,n-1),(2,n-1)\}+(t,t)\in M_1$ for $t=1,\ldots,m-2$. Therefore, $M_1\supseteq S$ (See \cref{CW(m;n): C and S; n<=2m}).

We show that $\G[M_1]$ has an alternating cycle leading us to a contradiction. Starting from the vertex $v_1:=(1,n)$ of $V(M_1)$, we define a sequence of vertices $v_1,v_2,\ldots$ as follows: if $v_i=(a,b)\notin S$, then 
\[v_{i+1}:=\begin{cases}
(a,b+1),&\{(a,b),(a,b+1)\}\in M_1,\\
(a,b+1),&\{(a,b-1),(a,b)\},\ \{(a,b+1),(a,b+2)\}\in M_1,\\
(a,b+1),&\{(a,b-1),(a,b)\}\in M_1,\ a=m,\\
(a+1,b),&\{(a,b-1),(a,b)\}\in M_1,\ \{(a,b+1),(a,b+2)\}\notin M_1,\ a<m,\\
(a-1,b),&\{(a-1,b),(a,b)\}\in M_1,\\
(a,b+1),&\{(a,b),(a+1,b)\}\in M_1
\end{cases}\]
and, if $v_i=(a,b)\in S$, then
\[v_{i+1}:=\begin{cases}
(a-1,b),&\{(a-1,b),(a,b)\}\in M_1,\\
(a,b-1),&\{(a,b),(a+1,b)\}\in M_1,\ a>1,\\
(a,b+1),&\{(a,b),(a+1,b)\}\in M_1,\ a=1.
\end{cases}\]

Conditions (1)--(4) on $M_1$ guarantee that $v_k=v_1$ for some $k$ and that the cycle $v_1,\ldots,v_k=v_1$ is an alternating cycle in $\G[M_1]$, a contradiction. Therefore, $\pmd(\G)=4$. The proof is complete.
\end{proof}

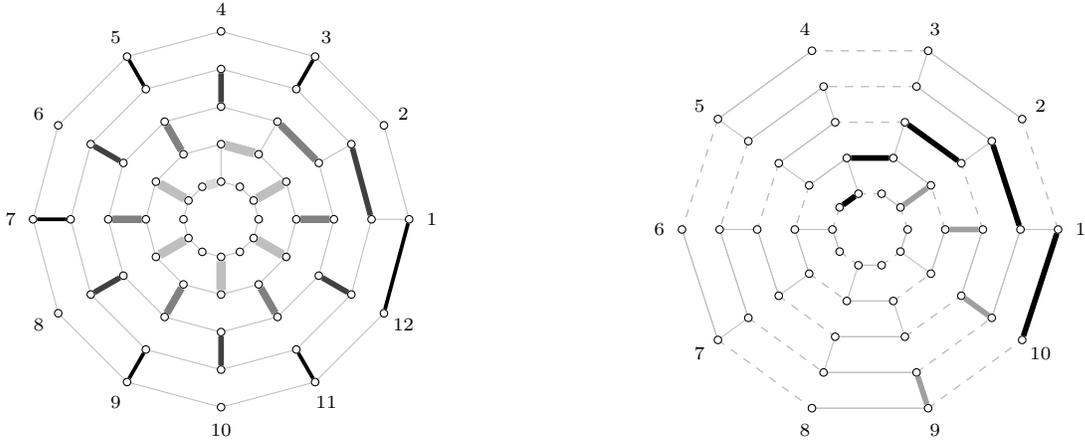
\begin{figure}[h]
\centering
\begin{subfigure}{0.45\textwidth}
\begin{tikzpicture}[scale=0.5]
\tikzmath{
function vertex(\c, \p){
	return int(mod(\c - 1, \n) + 1 + (\p - 1) * \n);
};
\m = 5;
\mm = int(\m / 2);
\n = 12;
\nn = \n / 2;
}

\foreach \c in {1,...,\n}
{
	\foreach \p in {1,...,\m}
	{
		\draw [gray!50] ({\p*cos(\c*360/\n)}, {\p*sin(\c*360/\n)})--({\p*cos((\c+1)*360/\n)}, {\p*sin((\c+1)*360/\n)});
	}
}

\foreach \c in {1,...,\nn}
{
	\foreach \p in {1,...,\mm}
	{
		\draw [gray!50] ({2*\p*cos(2*\c*360/\n)}, {2*\p*sin(2*\c*360/\n)})--({(2*\p+1)*cos(2*\c*360/\n)}, {(2*\p+1)*sin(2*\c*360/\n)});
		\draw [gray!50] ({(2*\p-1)*cos((2*\c+1)*360/\n)}, {(2*\p-1)*sin((2*\c+1)*360/\n)})--({2*\p*cos((2*\c+1)*360/\n)}, {2*\p*sin((2*\c+1)*360/\n)});
	}
}

\foreach \c in {1,...,\n}
{
	\foreach \p in {1,...,\m}
	{
		\tikzmath{\k = vertex(\c,\m + 1 - \p);}
		\node [draw, fill=white, circle, inner sep=1pt] (\k) at ({\p*cos((\c-1)*360/\n)}, {\p*sin((\c-1)*360/\n)}) {};
	}
	
	\node () at ({(\m + 0.6)*cos((\c-1)*360/\n)}, {(\m+0.6)*sin((\c-1)*360/\n)}) {\tiny{\c}};
}

\foreach \t in {0,...,3}
{
	\def\matching{{
	{{12,1},{1,1}},
	{{3,1},{3,2}},
	{{5,1},{5,2}},
	{{7,1},{7,2}},
	{{9,1},{9,2}},
	{{11,1},{11,2}}
	}}

	\foreach \i in {0,...,\m}
	{
		\tikzmath{\u = vertex(\matching[\i][0][0]+\t,\matching[\i][0][1]+\t);
		\v = vertex(\matching[\i][1][0]+\t,\matching[\i][1][1]+\t);
		\contrast = 100 - 25 * \t;
		\thickness = 0.5 + 25 * \t / 100;}
		\draw [line width=\thickness mm, black!\contrast] (\u)--(\v);
	}
}

\tikzmath{\u = vertex(4,5);\v = vertex(5,5);}
\draw [line width=1.25 mm, black!15] (\u)--(\v);
\end{tikzpicture}
\caption{Edge set $M$ (back) and its shifts (dark gray to light gray) ($n\geq2m+2$)}\label{CW(m;n): M1; n>=2m+2}
\end{subfigure}
\hfill
\begin{subfigure}{0.45\textwidth}
\centering
\begin{tikzpicture}[scale=0.5]
\tikzmath{
function vertex(\c, \p){
	return int(mod(\c - 1, \n) + 1 + (\p - 1) * \n);
};
\m = 5;
\mm = int(\m / 2);
\n = 10;
\nn = \n / 2;
}

\foreach \c in {1,...,\n}
{
	\foreach \p in {1,...,\m}
	{
		\draw [gray!50] ({\p*cos((2*\c-1)*360/\n)}, {\p*sin((2*\c-1)*360/\n)})--({\p*cos((2*\c)*360/\n)}, {\p*sin((2*\c)*360/\n)});
		\draw [dashed, gray!50] ({\p*cos((2*\c)*360/\n)}, {\p*sin((2*\c)*360/\n)})--({\p*cos((2*\c+1)*360/\n)}, {\p*sin((2*\c+1)*360/\n)});
	}
}

\foreach \c in {1,...,\nn}
{
	\foreach \p in {1,...,\mm}
	{
		\draw [gray!50] ({2*\p*cos(2*\c*360/\n)}, {2*\p*sin(2*\c*360/\n)})--({(2*\p+1)*cos(2*\c*360/\n)}, {(2*\p+1)*sin(2*\c*360/\n)});
		\draw [gray!50] ({(2*\p-1)*cos((2*\c+1)*360/\n)}, {(2*\p-1)*sin((2*\c+1)*360/\n)})--({2*\p*cos((2*\c+1)*360/\n)}, {2*\p*sin((2*\c+1)*360/\n)});
	}
}

\foreach \c in {1,...,\n}
{
	\foreach \p in {1,...,\m}
	{
		\tikzmath{\k = vertex(\c,\m + 1 - \p);}
		\node [draw, fill=white, circle, inner sep=1pt] (\k) at ({\p*cos((\c-1)*360/\n)}, {\p*sin((\c-1)*360/\n)}) {};
	}
	
	\node () at ({(\m + 0.6)*cos((\c-1)*360/\n)}, {(\m+0.6)*sin((\c-1)*360/\n)}) {\tiny{\c}};
}

\def\matching{{
{{10,1},{1,1}},
{{1,2},{2,2}},
{{2,3},{3,3}},
{{3,4},{4,4}},
{{4,5},{5,5}},
}}

\foreach \i in {0,...,4}
{
	\tikzmath{\u = vertex(\matching[\i][0][0],\matching[\i][0][1]);
	\v = vertex(\matching[\i][1][0],\matching[\i][1][1]);}
	\draw [line width=0.75 mm] (\u)--(\v);
}

\def\matching{{
{{9,1},{9,2}},
{{10,2},{10,3}},
{{1,3},{1,4}},
{{2,4},{2,5}},
}}

\foreach \i in {0,...,3}
{
	\tikzmath{\u = vertex(\matching[\i][0][0],\matching[\i][0][1]);
	\v = vertex(\matching[\i][1][0],\matching[\i][1][1]);}
	\draw [line width=0.75 mm, gray!75] (\u)--(\v);
}

\end{tikzpicture}
\caption{Edge sets $C$ (black) and $S$ (gray), and forbidden edges (dashed) ($n\leq2m$)}\label{CW(m;n): C and S; n<=2m}
\end{subfigure}
\caption{Positive matching $M_1$}
\end{figure}

We end up this paper by computing the pmd of the Cartesian product of two cycles completing the analysis of grid graphs.
\begin{theorem}\label{pmd(Cm O Cn)}
For all $m,n\geq3$, 
\[5\leq\pmd(C_m\square C_n)\leq6.\]
In addition, $\pmd(C_m\square C_n)=5$ if $m+n$ is odd and $\{m,n\}\neq\{3,4\},\{3,6\},\{5,6\}$.
\end{theorem}
\begin{proof}
Let $\G=C_m\square C_n$ and $V(\G)=[m]\times[n]$. Since $\G$ is a $4$-regular graph, we know from \cref{Regular graphs} that $\pmd(\G)\geq5$. We have three cases to consider:

(I) $m$ and $n$ have different parities, say $m$ is odd and $n$ is even. First assume that $n\geq8$. Let
\begin{align*}
M_1:=\{&\{(2i-1,2j-1),(2i,2j-1)\}\colon 2i\leq m-1, 2j\leq n\}\\
&\cup\{\{(2i,2j),(2i+1,2j)\}\colon 2i\leq m-1, 2j\leq n\}
\end{align*}
and 
\begin{align*}
M_2:=\{&\{(2i-1,2j),(2i,2j)\}\colon 2i\leq m-1, 2j\leq n\}\\
&\cup\{\{(2i,2j-1),(2i+1,2j-1)\}\colon 2i\leq m-3, 2j\leq n\}\\
&\cup\{\{(1,2j-1),(m,2j-1)\}\colon 2j\leq n\}.
\end{align*}
Clearly, $M_1$ and $M_2$ are positive matchings in $\G$ and 
\[\G-M_1-M_2\cong CW(3, n)\cup (m-3)C_n\]
(see \cref{Cm O Cn: M1+M2; m odd; n>=8}). By \cref{pmd of circular wall graphs}, 
\[\pmd(\G)\leq2+\pmd(CW(3, n))=5\]
Thus $\pmd(\G)=5$. 

Now, assume that $n\leq 6$. Suppose $m\geq n+1$ that is $(m,n)\neq(3,4),(3,6),(5,6)$. Let
\begin{align*}
M_1:=\{&\{(2i+j-2,2i+j-2),(2i+j-1,2i+j-2)\}\colon i\leq(m-1)/2, j\leq n\}\\
&\setminus\{\{(2i-1,n),(2i,n)\}\colon 2i<n\},
\end{align*}
and
\begin{align*}
M_2:=\{&\{(2i+j,2i+j-1),(2i+j+1,2i+j-1)\}\colon i\leq(m-1)/2, j\leq n\}\\
&\setminus\{\{(2i-1,1),(2i,1)\}\colon 2i\leq n\}.
\end{align*}
Clearly, $M_1$ and $M_2$ are positive matchings in $\G$ (see  \cref{Cm O Cn: M1+M2+M3; m odd; n even; m>=n+1}). 
If 
\[M_3:=\{(i,1),(i,n)\}\colon i\in[m]\}\cup\{(i,i+1),(i+1,i+1)\}\colon i\in[n-2]\},\]
then $M_3$ is a positive matching in $\G-M_1-M_2$, and 
\[\G-M_1-M_2-M_3\cong P_{n(n+2)}\cup (m-n-2)P_n\]
(see \cref{Cm O Cn: M1+M2+M3; m odd; n even; m>=n+1}). Thus $\pmd(\G)\leq5$, which implies that $\pmd(\G)=5$.

(II) $m$ and $n$ are even. Assume $m\leq n$ and put
\begin{align*}
M_1:=\{&\{(2i-1,2j-1),(2i,2j-1)\}\colon 2i\leq m-2, 2j\leq n\}\\
&\cup\{\{(2i,2j),(2i+1,2j)\}\colon 2i\leq m-4, 2j\leq n\}\\
&\cup\{\{(m-1,2j),(m,2j)\colon 2j\leq n\}
\end{align*}
and 
\begin{align*}
M_2:=\{&\{(2i-1,2j),(2i,2j)\}\colon 2i\leq m-2, 2j\leq n\}\\
&\cup\{\{(2i,2j-1),(2i+1,2j-1)\}\colon 2i\leq m-4, 2j\leq n\}\\
&\cup\{\{(m-1,2j-1),(m,2j-1)\}\colon 2j\leq n\}.
\end{align*}
Clearly, $M_1$ and $M_2$ are positive matchings in $\G$ and 
\[\G-M_1-M_2\cong 2(P_2\square C_n)\cup(m-4)C_n\]
(see \cref{Cm O Cn: M1+M2; m even; n even; 4<=m<=n}). By \cite[Proposition 4.1]{mfdg-sg-aayp}, 
\[\pmd(\G)\leq2+\pmd(P_2\square C_n)=6\]
except for $n=4$. If $n=4$, then $m=4$ and we get from \cref{Cm O Cn: M1+M2+M3+M4; m even; n even; m=n=4} that $\pmd(\G)\leq6$ for $\G-M_1-M_2-M_3-M_4\cong P_8\cup2P_4$.

(III) $m$ and $n$ are odd. If $m=n=3$, then one can simply find a pmd of size $6$ so that $\pmd(\G)\leq6$. Hence, assume that $m\geq5$. Let 
\begin{align*}
M_1:=\{&\{(2i-1,2j-1),(2i,2j-1)\}\colon 2i\leq m-1, 2j\leq n-1\}\\
&\cup\{\{(2i,2j),(2i+1,2j)\}\colon 2i\leq m-1, 2j\leq n-1\}
\end{align*}
and
\begin{align*}
M_2:=\{&\{(2i-1,2j),(2i,2j)\}\colon 2i\leq m-1, 2j\leq n-1\}\\
&\cup\{\{(2i,2j-1),(2i+1,2j-1)\}\colon 2i\leq m-1, 2j\leq n-1\}.
\end{align*}
Clearly, $M_1$ and $M_2$ are positive matchings in $\G$ (see \cref{Cm O Cn: M1+M2; m odd; n odd; m>=5; n>=3}). Now, if
\begin{align*}
M_3=\{&\{(1,2j-1),(1,2j)\}\colon 2j\leq n-1\}\\
&\cup\{\{(m,2j),(m,2j+1)\}\colon 2j\leq n-3\}\\
&\cup\{\{(i,1),(i,n)\}\colon 2\leq i\leq n-2\}\cup\{\{(1,n),(m,n)\}\}
\end{align*}
and
\begin{align*}
M_4=\{&\{(1,2j-2),(1,2j-1)\}\colon 2j\leq n-1\}\\
&\cup\{\{(m,2j-1),(m,2j)\}\colon 2j\leq n-1\}\\
&\cup\{\{(i,n-1),(i,n)\}\colon 2\leq i\leq n-2\}\cup\{\{(m-1,n),(m,n)\}\},
\end{align*}
then $M_3$ and $M_4$ are positive matchings in $\G-M_1-M_2$ (see \cref{Cm O Cn: M3+M4; m odd; n odd; m>=5; n>=3}), and 
\[\G-M_1-M_2-M_3-M_4\cong P_{m+n+3}\cup(m-3)P_{n-1}\cup(n-3)P_2.\]
Thus $\pmd(\G)\leq6$.
\end{proof}

\begin{figure}[h]
\centering
\begin{subfigure}{0.45\textwidth}
\centering
\begin{tikzpicture}[scale=0.5]
\tikzmath{\m = 9;\n = 8;\nn = \n / 2;}

\draw [dotted] (-0.5,-0.5) rectangle (\n-0.5,\m-0.5);

\foreach \row in {1,...,\m}
	\draw [gray!50] (-0.5,\row-1)--(\n-0.5,\row-1);

\foreach \col in {1,...,\n}
	\draw [gray!50] (\col-1,-0.5)--(\col-1,\m-0.5);

\tikzmath{\rows=int((\m-1)/2);\cols=int(\n/2);}

\foreach \row in {1,...,\rows}
	\foreach \col in {1,...,\cols}
		\draw [line width=0.75 mm] (2*\col-2,2*\row-1)--(2*\col-2,2*\row-2);

\foreach \row in {1,...,\rows}
	\foreach \col in {1,...,\cols}
		\draw [line width=0.75 mm] (2*\col-1,2*\row-1)--(2*\col-1,2*\row);

\tikzmath{\rows=int((\m-1)/2);\cols=int(\n/2);}

\foreach \row in {1,...,\rows}
	\foreach \col in {1,...,\cols}
		\draw [line width=0.75 mm, gray!75] (2*\col-1,2*\row-1)--(2*\col-1,2*\row-2);

\tikzmath{\rows=int((\m-3)/2);\cols=int(\n/2);}

\foreach \row in {1,...,\rows}
	\foreach \col in {1,...,\cols}
		\draw [line width=0.75 mm, gray!75] (2*\col-2,2*\row-1)--(2*\col-2,2*\row);

\foreach \col in {1,...,\cols}
{
	\draw [line width=0.75 mm, gray!75] (2*\col-2,-0.5)--(2*\col-2,0);
	\draw [line width=0.75 mm, gray!75] (2*\col-2,\m-0.5)--(2*\col-2,\m-1);
}

\foreach \row in {1,...,\m}
	\foreach \col in {1,...,\n}
		\tikzmath{\x=\col-1;\y={mod(2*\m-1-\row, \m)};}
		\node [draw, fill=white, circle, inner sep=1pt] () at (\x, \y) {};

\foreach \row in {1,...,\m}
	\tikzmath{\y={mod(2*\m-1-\row, \m)};}
	\node () at (-1, \y) {\tiny\row};

\foreach \col in {1,...,\n}
	\node () at (\col-1, \m) {\tiny\col};
\end{tikzpicture}
\caption{Matchings $M_1$ (black) and $M_2$ (gray) ($n\geq8$)}\label{Cm O Cn: M1+M2; m odd; n>=8}
\end{subfigure}
\hfill
\begin{subfigure}{0.45\textwidth}
\centering
\begin{tikzpicture}[scale=0.5]
\tikzmath{\m = 9;\n = 6;\nn = \n / 2;}

\draw [dotted] (-0.5,-0.5) rectangle (\n-0.5,\m-0.5);

\draw [gray!50] (0,0)--(0,\m-1);
\draw [gray!50] (\n-1,0)--(\n-1,\m-1);

\foreach \row in {1,...,\m}
	\draw [gray!50] (0,\row-1)--(\n-1,\row-1);

\foreach \i in {1,...,\nn}
{
	\draw [line width=0.75 mm] (2*\i-1,0)--(2*\i-1,-0.5);
	\draw [line width=0.75 mm] (2*\i-1,\m-1)--(2*\i-1,\m-0.5);
}

\def\matching{{
{1,1},{2,2},{3,3},{4,4},{5,5},{6,6},
{3,1},{4,2},{5,3},{6,4},{7,5},
{5,1},{6,2},{7,3},{9,5},
{7,1},{9,3},{1,4},{2,5}
}}

\foreach \i in {0,...,18}
{
	\tikzmath{\x=\matching[\i][1]-1;\y={mod(2*\m-1-\matching[\i][0], \m)};}
	\draw [line width=0.75 mm] (\x,\y)--(\x,\y-1);
}

\foreach \i in {1,...,\nn}
{
	\draw [line width=0.75 mm, gray!75] (2*\i-2,0)--(2*\i-2,-0.5);
	\draw [line width=0.75 mm, gray!75] (2*\i-2,\m-1)--(2*\i-2,\m-0.5);
}

\def\matching{{
{3,2},{4,3},{5,4},{6,5},{7,6},
{5,2},{6,3},{7,4},{9,6},
{7,2},{9,4},{1,5},{2,6},
{9,2},{1,3},{2,4},{3,5},{4,6}
}}

\foreach \i in {0,...,17}
{
	\tikzmath{\x=\matching[\i][1]-1;\y={mod(2*\m-1-\matching[\i][0], \m)};}
	\draw [line width=0.75 mm, gray!75] (\x,\y)--(\x,\y-1);
}

\foreach \i in {0,...,8}
{
	\draw [line width=0.25 mm, gray, decorate, decoration={snake, segment length=1mm, amplitude=0.5mm}] (-0.5,\i)--(0,\i);
	\draw [line width=0.25 mm, gray, decorate, decoration={snake, segment length=1mm, amplitude=0.5mm}] (\n-1,\i)--(\n-0.5,\i);
}

\foreach \i in {3,...,\n}
	\draw [line width=0.25 mm, gray, decorate, decoration={snake, segment length=1mm, amplitude=0.5mm}] (\i-2,\m-\i)--(\i-2,\m-\i+1);
	
\foreach \row in {1,...,\m}
	\foreach \col in {1,...,\n}
		\tikzmath{\x=\col-1;\y={mod(2*\m-1-\row, \m)};}
		\node [draw, fill=white, circle, inner sep=1pt] () at (\x, \y) {};

\foreach \row in {1,...,\m}
	\tikzmath{\y={mod(2*\m-1-\row, \m)};}
	\node () at (-1, \y) {\tiny\row};

\foreach \col in {1,...,\n}
	\node () at (\col-1, \m) {\tiny\col};
\end{tikzpicture}
\caption{Matchings $M_1$ (black), $M_2$ (gray), and $M_3$ (curved) ($m\geq n+1$)}\label{Cm O Cn: M1+M2+M3; m odd; n even; m>=n+1}
\end{subfigure}
\caption{Matchings $M_1$, $M_2$, and $M_3$ ($m$ odd and $n$ even)}
\end{figure}
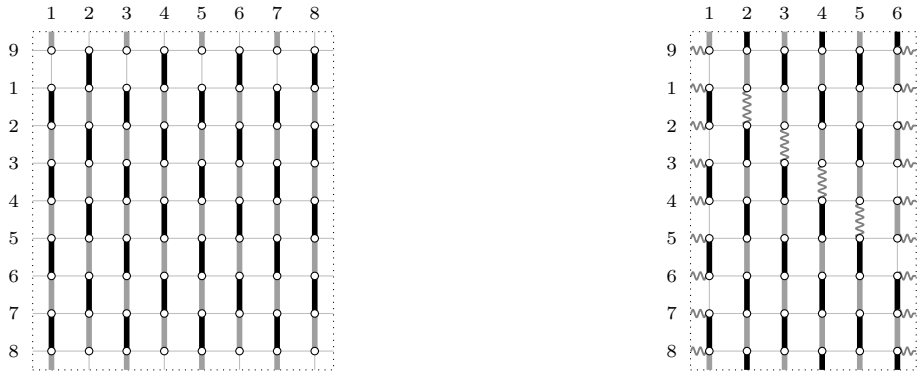

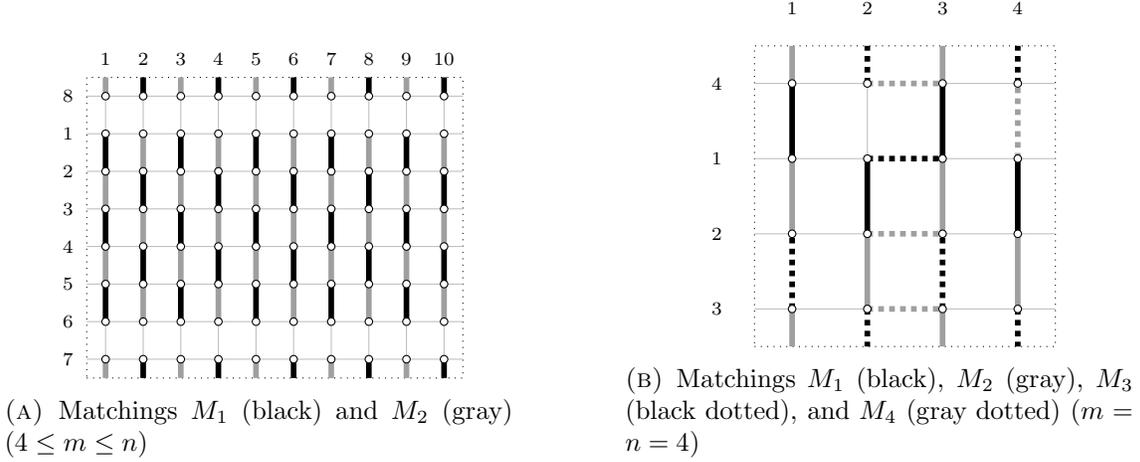
\begin{figure}[h]
\centering
\begin{subfigure}{0.45\textwidth}
\centering
\begin{tikzpicture}[scale=0.5]
\tikzmath{\m = 8;\n = 10;\nn = \n / 2;}

\draw [dotted] (-0.5,-0.5) rectangle (\n-0.5,\m-0.5);

\foreach \row in {1,...,\m}
	\draw [gray!50] (-0.5,\row-1)--(\n-0.5,\row-1);

\foreach \col in {1,...,\n}
	\draw [gray!50] (\col-1,-0.5)--(\col-1,\m-0.5);

\tikzmath{\rows=int((\m-2)/2);\cols=int(\n/2);}

\foreach \row in {1,...,\rows}
	\foreach \col in {1,...,\cols}
		\draw [line width=0.75 mm] (2*\col-2,2*\row)--(2*\col-2,2*\row-1);

\foreach \row in {2,...,\rows}
	\foreach \col in {1,...,\cols}
		\draw [line width=0.75 mm] (2*\col-1,2*\row-1)--(2*\col-1,2*\row-2);

\foreach \col in {1,...,\cols}
{
	\draw [line width=0.75 mm] (2*\col-1,-0.5)--(2*\col-1,0);
	\draw [line width=0.75 mm] (2*\col-1,\m-0.5)--(2*\col-1,\m-1);
}				
\tikzmath{\rows=int((\m-1)/2);\cols=int(\n/2);}

\foreach \row in {1,...,\rows}
	\foreach \col in {1,...,\cols}
		\draw [line width=0.75 mm, gray!75] (2*\col-1,2*\row)--(2*\col-1,2*\row-1);

\foreach \row in {2,...,\rows}
	\foreach \col in {1,...,\cols}
		\draw [line width=0.75 mm, gray!75] (2*\col-2,2*\row-1)--(2*\col-2,2*\row-2);

\foreach \col in {1,...,\cols}
{
	\draw [line width=0.75 mm, gray!75] (2*\col-2,-0.5)--(2*\col-2,0);
	\draw [line width=0.75 mm, gray!75] (2*\col-2,\m-0.5)--(2*\col-2,\m-1);
}

\foreach \row in {1,...,\m}
	\foreach \col in {1,...,\n}
		\tikzmath{\x=\col-1;\y={mod(2*\m-1-\row, \m)};}
		\node [draw, fill=white, circle, inner sep=1pt] () at (\x, \y) {};

\foreach \row in {1,...,\m}
	\tikzmath{\y={mod(2*\m-1-\row, \m)};}
	\node () at (-1, \y) {\tiny\row};

\foreach \col in {1,...,\n}
	\node () at (\col-1, \m) {\tiny\col};
\end{tikzpicture}
\caption{Matchings $M_1$ (black) and $M_2$ (gray) ($4\leq m\leq n$)}\label{Cm O Cn: M1+M2; m even; n even; 4<=m<=n}
\end{subfigure}
\hfill
\begin{subfigure}{0.45\textwidth}
\centering
\begin{tikzpicture}[scale=1]
\tikzmath{\m = 4;\n = 4;\nn = \n / 2;}

\draw [dotted] (-0.5,-0.5) rectangle (\n-0.5,\m-0.5);

\foreach \row in {1,...,\m}
{
	\draw [gray!50] (-0.5,\row-1)--(1,\row-1);
	\draw [gray!50] (2,\row-1)--(\n-0.5,\row-1);
}

\draw [gray!50] (1,2)--(1,3);

\draw [line width=0.75 mm] (0,2)--(0,3);
\draw [line width=0.75 mm] (1,1)--(1,2);
\draw [line width=0.75 mm] (2,2)--(2,3);
\draw [line width=0.75 mm] (3,1)--(3,2);

\draw [line width=0.75 mm, gray!75] (0,1)--(0,2);
\draw [line width=0.75 mm, gray!75] (1,0)--(1,1);
\draw [line width=0.75 mm, gray!75] (2,1)--(2,2);
\draw [line width=0.75 mm, gray!75] (3,0)--(3,1);
\draw [line width=0.75 mm, gray!75] (0,0)--(0,-0.5);
\draw [line width=0.75 mm, gray!75] (0,3)--(0,3.5);
\draw [line width=0.75 mm, gray!75] (2,0)--(2,-0.5);
\draw [line width=0.75 mm, gray!75] (2,3)--(2,3.5);

\draw [line width=0.75 mm, dotted] (0,0)--(0,1);
\draw [line width=0.75 mm, dotted] (2,0)--(2,1);
\draw [line width=0.75 mm, dotted] (1,2)--(2,2);
\draw [line width=0.75 mm, dotted] (1,0)--(1,-0.5);
\draw [line width=0.75 mm, dotted] (1,3)--(1,3.5);
\draw [line width=0.75 mm, dotted] (3,0)--(3,-0.5);
\draw [line width=0.75 mm, dotted] (3,3)--(3,3.5);

\draw [line width=0.75 mm, dotted, gray!75] (1,0)--(2,0);
\draw [line width=0.75 mm, dotted, gray!75] (1,1)--(2,1);
\draw [line width=0.75 mm, dotted, gray!75] (1,3)--(2,3);
\draw [line width=0.75 mm, dotted, gray!75] (3,2)--(3,3);

	
\foreach \row in {1,...,\m}
	\foreach \col in {1,...,\n}
		\tikzmath{\x=\col-1;\y={mod(2*\m-1-\row, \m)};}
		\node [draw, fill=white, circle, inner sep=1pt] () at (\x, \y) {};

\foreach \row in {1,...,\m}
	\tikzmath{\y={mod(2*\m-1-\row, \m)};}
	\node () at (-1, \y) {\tiny\row};

\foreach \col in {1,...,\n}
	\node () at (\col-1, \m) {\tiny\col};
\end{tikzpicture}
\caption{Matchings $M_1$ (black), $M_2$ (gray), $M_3$ (black dotted), and $M_4$ (gray dotted) ($m=n=4$)}\label{Cm O Cn: M1+M2+M3+M4; m even; n even; m=n=4}
\end{subfigure}
\caption{Matchings $M_1$, $M_2$, $M_3$, and $M_4$ ($m$ and $n$ are both even)}
\end{figure}

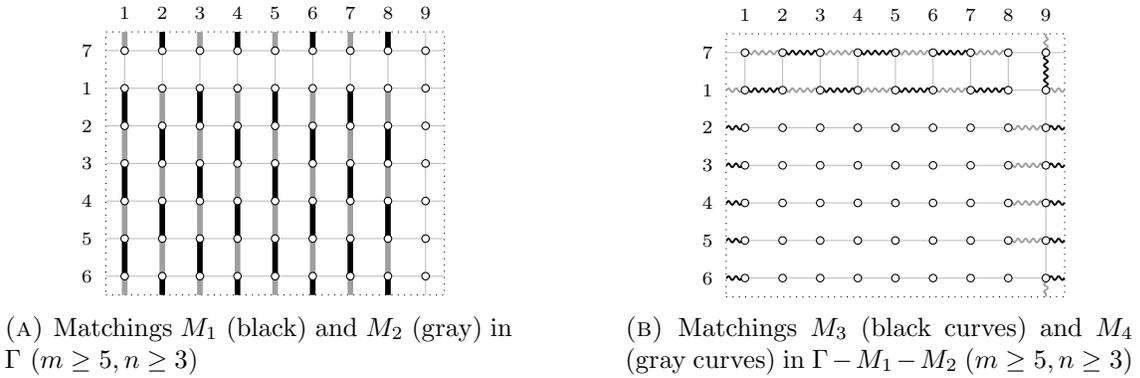
\begin{figure}[h]
\centering
\begin{subfigure}{0.45\textwidth}
\centering
\begin{tikzpicture}[scale=0.5]
\tikzmath{\m = 7;\n = 9;\nn = (\n - 1) / 2;}

\draw [dotted] (-0.5,-0.5) rectangle (\n-0.5,\m-0.5);

\foreach \row in {1,...,\m}
	\draw [gray!50] (-0.5,\row-1)--(\n-0.5,\row-1);

\foreach \col in {1,...,\n}
	\draw [gray!50] (\col-1,-0.5)--(\col-1,\m-0.5);

\tikzmath{\rows=int((\m-1)/2);\cols=int(\n/2);}

\foreach \row in {1,...,\rows}
	\foreach \col in {1,...,\cols}
		\draw [line width=0.75 mm] (2*\col-2,2*\row-1)--(2*\col-2,2*\row-2);

\foreach \row in {2,...,\rows}
	\foreach \col in {1,...,\cols}
		\draw [line width=0.75 mm] (2*\col-1,2*\row-2)--(2*\col-1,2*\row-3);

\foreach \col in {1,...,\cols}
{
	\draw [line width=0.75 mm] (2*\col-1,-0.5)--(2*\col-1,0);
	\draw [line width=0.75 mm] (2*\col-1,\m-0.5)--(2*\col-1,\m-1);
}				
\tikzmath{\rows=int((\m-1)/2);\cols=int(\n/2);}

\foreach \row in {1,...,\rows}
	\foreach \col in {1,...,\cols}
		\draw [line width=0.75 mm, gray!75] (2*\col-1,2*\row-1)--(2*\col-1,2*\row-2);

\foreach \row in {2,...,\rows}
	\foreach \col in {1,...,\cols}
		\draw [line width=0.75 mm, gray!75] (2*\col-2,2*\row-2)--(2*\col-2,2*\row-3);

\foreach \col in {1,...,\cols}
{
	\draw [line width=0.75 mm, gray!75] (2*\col-2,-0.5)--(2*\col-2,0);
	\draw [line width=0.75 mm, gray!75] (2*\col-2,\m-0.5)--(2*\col-2,\m-1);
}

\foreach \row in {1,...,\m}
	\foreach \col in {1,...,\n}
		\tikzmath{\x=\col-1;\y={mod(2*\m-1-\row, \m)};}
		\node [draw, fill=white, circle, inner sep=1pt] () at (\x, \y) {};

\foreach \row in {1,...,\m}
	\tikzmath{\y={mod(2*\m-1-\row, \m)};}
	\node () at (-1, \y) {\tiny\row};

\foreach \col in {1,...,\n}
	\node () at (\col-1, \m) {\tiny\col};
\end{tikzpicture}
\caption{Matchings $M_1$ (black) and $M_2$ (gray) in $\G$ ($m\geq5,n\geq3$)}\label{Cm O Cn: M1+M2; m odd; n odd; m>=5; n>=3}
\end{subfigure}
\hfill
\begin{subfigure}{0.45\textwidth}
\centering
\begin{tikzpicture}[scale=0.5]
\tikzmath{\m = 7;\n = 9;\nn = (\n - 1) / 2;}

\draw [dotted] (-0.5,-0.5) rectangle (\n-0.5,\m-0.5);

\foreach \row in {4,...,\m}
{
	\draw [gray!50] (0,\row-3)--(\n-2,\row-3);
}

\draw [gray!50] (0,0)--(\n-1,0);
\draw [gray!50] (\n-2,\m-2)--(\n-1,\m-2);
\draw [gray!50] (\n-2,\m-1)--(\n-0.5,\m-1);
\draw [gray!50] (-0.5,\m-1)--(0,\m-1);

\foreach \col in {2,...,\n}
	\draw [gray!50] (\col-2,\m-1)--(\col-2,\m-2);

\draw [gray!50] (\n-1,0)--(\n-1,\m-2);


\draw [line width=0.25 mm, decorate, decoration={snake, segment length=1mm, amplitude=0.25mm}] (\n-1,\m-2)--(\n-1,\m-1);

\foreach \col in {2,...,\nn}
{
	\draw [line width=0.25 mm, decorate, decoration={snake, segment length=1mm, amplitude=0.25mm}] (2*\col-3,\m-1)--(2*\col-2,\m-1);
}

\foreach \col in {1,...,\nn}
{
	\draw [line width=0.25 mm, decorate, decoration={snake, segment length=1mm, amplitude=0.25mm}] (2*\col-2,\m-2)--(2*\col-1,\m-2);
}

\foreach \row in {3,...,\m}
{
	\draw [line width=0.25 mm, decorate, decoration={snake, segment length=1mm, amplitude=0.25mm}] (-0.5,\row-3)--(0,\row-3);
	\draw [line width=0.25 mm, decorate, decoration={snake, segment length=1mm, amplitude=0.25mm}] (\n-0.5,\row-3)--(\n-1,\row-3);
}

\foreach \col in {1,...,\nn}
{
	\draw [line width=0.25 mm, gray!75, decorate, decoration={snake, segment length=1mm, amplitude=0.25mm}] (2*\col-2,\m-1)--(2*\col-1,\m-1);
}

\foreach \col in {2,...,\nn}
{
	\draw [line width=0.25 mm, gray!75, decorate, decoration={snake, segment length=1mm, amplitude=0.25mm}] (2*\col-3,\m-2)--(2*\col-2,\m-2);
}

\draw [line width=0.25 mm, gray!75, decorate, decoration={snake, segment length=1mm, amplitude=0.25mm}] (-0.5,\m-2)--(0,\m-2);
	\draw [line width=0.25 mm, gray!75, decorate, decoration={snake, segment length=1mm, amplitude=0.25mm}] (\n-0.5,\m-2)--(\n-1,\m-2);

\foreach \row in {4,...,\m}
{
	\draw [line width=0.25 mm, gray!75, decorate, decoration={snake, segment length=1mm, amplitude=0.25mm}] (\n-2,\row-3)--(\n-1,\row-3);
}

\draw [line width=0.25 mm, gray!75, decorate, decoration={snake, segment length=1mm, amplitude=0.25mm}] (\n-1,0)--(\n-1,-0.5);
\draw [line width=0.25 mm, gray!75, decorate, decoration={snake, segment length=1mm, amplitude=0.25mm}] (\n-1,\m-1)--(\n-1,\m-0.5);

\foreach \row in {1,...,\m}
	\foreach \col in {1,...,\n}
		\tikzmath{\x=\col-1;\y={mod(2*\m-1-\row, \m)};}
		\node [draw, fill=white, circle, inner sep=1pt] () at (\x, \y) {};

\foreach \row in {1,...,\m}
	\tikzmath{\y={mod(2*\m-1-\row, \m)};}
	\node () at (-1, \y) {\tiny\row};

\foreach \col in {1,...,\n}
	\node () at (\col-1, \m) {\tiny\col};
\end{tikzpicture}
\caption{Matchings $M_3$ (black curves) and $M_4$ (gray curves) in $\G-M_1-M_2$ ($m\geq5,n\geq3$)}\label{Cm O Cn: M3+M4; m odd; n odd; m>=5; n>=3}
\end{subfigure}
\caption{Matchings $M_1$, $M_2$, $M_3$, and $M_4$ ($m$ and $n$ are both odd)}
\end{figure}

The following proposition shows that the upper bound in Theorem \ref{pmd(Cm O Cn)} is also sharp. To prove $\pmd(C_4\square C_4)=6$ we use the notion of Cayley graphs of groups. Recall that the Cayley graph $\Cay(G,C)$ of a group $G$ with respect to an inverse closed subset $C$ of $G$ is the graph with vertex set $G$ and edges $\{g,gc\}$ for all $g\in G$ and $c\in C$. It turns out that $Q_n\cong\Cay(\oplus_{i=1}^n\ZZ_2,\{\e_1,\ldots,\e_n\})$ in which $\e_1,\ldots,\e_n$ are the standard basis elements of $\oplus_{i=1}^n\ZZ_2$ viewed as a vector space over $\ZZ_2$. In what follows, $\delta(\G)$ denotes the minimum degree of the graph $\G$.
\begin{proposition}\label{pmd(C3OC3)=pmd(C4OC4)=6}
$\pmd(C_3\square C_3)=\pmd(C_4\square C_4)=6$.
\end{proposition}
\begin{proof}
We know from the proof of Theorem \ref{pmd(Cm O Cn)} that $5\leq\pmd(C_k\square C_k)\leq6$ for $k=3,4$. We show that $\pmd(C_k\square C_k)\neq5$ for all $k=3,4$, from which the result follows.

(I) $\pmd(C_3\square C_3)\neq5$. Let $\G:=C_3\square C_3$ and assume on the contrary that $\G$ has a pmd $M_1,\ldots,M_5$ with $5$ parts. First observe that $|M_1|\leq2$ otherwise we obtain an alternating closed walk in $\G[M_1]$. Note that the first edge in $M_1$ can be fixed due to the symmetry of $\G$. Since $|\G|=9$ any matching in $\G$ covers at most eight vertices so that $|M_i|\leq4$ ($i=2,\ldots,5$). As $|E(\G-M_1)|\geq18-2=16$ we observe that $|M_1|=2$ and $|M_i|=4$ ($i=2,\ldots,5$). On the other hand, the fact that $\delta(\G-M_1-v)\geq2$ for every $v\in V(\G)$ yields $|M_2|<4$ as $(\G-M_1)[M_2]$ has a pendant by Theorem \ref{positivity of matchings}, a contradiction.

(II) $\pmd(C_4\square C_4)\neq5$. Let $\G=C_4\square C_4$ and suppose on the contrary that $\pmd(\G)=5$. Let $M_1,\ldots,M_5$ be a pmd of $\G$. We proceed as follows:
\begin{itemize}
\item[(1)]$|M_i\cup M_j|<16$ for all $i\neq j$.\\
Assume $|M_i\cup M_j|=16$. Then $|M_i|=|M_j|=8$. If $\G':=\G-M_1-\cdots-M_{i-1}$, then $\delta(\G')\geq2$ so that $\G'[M_i]$ has an alternating closed walk, a contradiction.
\item[(2)]$|M_i|<8$ for all $i<5$ (after a suitable relabeling of $M_i$'s).\\
If $|M_i|=8$ and $\G':=\G-M_1-\cdots-M_{i-1}$, then since $\G'[M_i]$ has a pendant, it follows that $i\geq4$. Thus $\G'$ is a union of paths so that we may swap $M_4$ and $M_5$ (if required) and assume that $|M_4|<8$.
\item[(3)]$|M_1\cup M_2|\geq11$.\\
Suppose on the contrary that $|M_1\cup M_2|\leq10$. Let $\G':=\G-M_1-M_2$. Then $|E(\G')|\geq22$. It follows that $|M_3|=|M_4|=7$ and $|M_5|=8$ by part (2). Thus $\G'-M_3$ is a Hamiltonian path. A simple verification, by checking all possible cases, shows that $\G'[M_3]$ has an alternating closed walk, which is a contradiction.
\item[(4)]$|M_1|=4$.\\
If $|M_1|\leq3$ then $|M_2|=8$ contradicting part (2). Thus $|M_1|\geq4$. First observe that $\G\cong Q_4\cong\Cay(G,\{\e_1,\e_2,\e_3,\e_4\})$ with $\e_1,\e_2,\e_3,\e_4$ being the standard basis for $G:=\ZZ_2\oplus\ZZ_2\oplus\ZZ_2\oplus\ZZ_2$ (as a vector space over $\ZZ_2$). In the following, a subgraph of $\G$ isomorphic to $Q_3$ is referred to as a cube subgraph of $\G$. Also, the sum $x+\{y,z\}$ of a vertex $x$ and an edge $\{y,z\}$ is simply the edge $\{x+y,x+z\}$ (in the given Cayley graph). Let $e_i:=\{0,\e_i\}$ be the edge connecting $0$ to $\e_i$, for all $i=1,2,3,4$. Utilizing the above observations, we get:
\begin{itemize}
\item[(4a)]Any two edges $e,e'\in M_1$ connected by an edge in $\G$ belong to a cube subgraph of $\G$.\\
Indeed, $e=x+e_i$ and $e'=x+\e_i+\e_j+e_k$ for some $i,j,k\in\{1,2,3,4\}$ and $x\in G$ with $x+\e_i+e_j$ being the edge connecting $e,e'$. Thus $e,e'$ belong to the cube subgraph with vertex set $x+\gen{\e_i,\e_j,\e_k}$, where $\gen{\e_i,\e_j,\e_k}$ denotes the subgroup of $G$ generated by $\e_i,\e_j,\e_k$. Notice that $i,j,k$ are distinct as $\{e,e'\}$ is a positive matching.
\item[(4b)]Given three edges $e,e',e''\in M_1$, there always exists a cube subgraph of $\G$ containing exactly two of $e,e',e''$.\\
First observe that no cube subgraph of $\G$ contains all of $e,e',e''$ as $\pmd(Q_3)=5$ and $Q_3$ has no positive matching of size three. If $e,e'$ do not belong to a cube subgraph of $\G$, then $e,e'$ are at distance at least two by part (a). Then $e''$ is either adjacent to $e$ or $e'$ from which it follows that either $e,e''$ or $e',e''$ belong to a cube subgraph of $\G$ by part (a).
\item[(4c)]$|M_1|<5$.\\
Suppose $|M_1|\geq5$. Since $|M_1|\geq3$, $M_1$ has two edges $e,e'$ belonging to a cube subgraph of $\G$, say $C$. Let $C'$ be the cube subgraph of $\G$ disjoint from $C$. Then $|M_1\cap E(C)|=2$ and $1\leq|M_1\cap E(C')|\leq2$. It follows that $M_1$ has an edge $e''$ connecting $C$ and $C'$. By symmetry of $C$, we have just two configurations for $e,e'$, and for any configuration we have at most two choices (modulo symmetry) for $e''$. Now, a simple manipulation shows that $\{e,e',e''\}$ cannot be extended to a positive matching with five edges contradicting the assumption that $|M_1|\geq5$.
\end{itemize}
Since $3<|M_1|<5$, it follows that $|M_1|=4$, as required.
\item[(5)]$|M_2|=7$.\\
It is obvious as $|M_1\cup M_2|\geq11$ and $|M_2|<8$.
\end{itemize}
Since $M_2$ is a positive matching in $\G-M_1$ and $\G-M_1-M_2$ has no vertices of degree $4$, there exist edges $uv$, $u'v'$, $u''v''$ in $M_1$ such that $u',u''\in N_\G(u)$ and $u',u''\notin V(M_2)$ leaving $u$ a pendant in $(\G-M_1)[M_2]$. Without loss of generality, we may assume that $u=0$, $v=\e_1$, $u'=\e_2$, $u''=\e_3$ as $v,u',u''\in N_\G(u)$ and lie on a cube subgraph of $\G$. A simple verification yields the following five possible configurations for $M_1$:
\begin{align*}
&\{e_1,\ \e_2+e_3,\ \e_3+e_4,\ \e_1+\e_4+e_2\},\\
&\{e_1,\ \e_2+e_3,\ \e_3+e_4,\ \e_1+\e_2+\e_4+e_3\},\\
&\{e_1,\ \e_2+e_3,\ \e_3+e_4,\ \e_1+\e_3+\e_4+e_2\},\\
&\{e_1,\ \e_2+e_4,\ \e_3+e_4,\ \e_1+\e_2+\e_4+e_3\},\\
&\{e_1,\ \e_2+e_4,\ \e_3+e_4,\ \e_1+\e_3+\e_4+e_2\}.
\end{align*}
Analyzing all possible matchings of size $7$ in $\G-M_1$, via a simple computer program, reveals that $\G-M_1$ has no positive matchings of size $7$. This contradiction completes the proof.
\end{proof}
\begin{conjecture}
$\pmd(C_m\square C_n)=6$ for all $m,n\geq3$ with $m+n$ even or $(m,n)=(3,4)$, $(3,6)$, $(5,6)$.
\end{conjecture}

\end{document}